\pdfoutput=1
\documentclass[11pt,reqno]{amsart}

\usepackage{amsfonts,amsmath,amssymb,amsthm}
\usepackage{bbm}
\usepackage{hyperref}
\usepackage[english]{babel}
\usepackage{parskip}
\usepackage{graphicx}
\graphicspath{ {images/} }
\usepackage{bm}
\usepackage{physics}
\usepackage{xcolor}
\usepackage{tikz}
\usepackage{polski}
\usepackage{subcaption}
\usepackage[format= plain, labelformat=default,labelsep=colon,skip=3pt]{caption}
\usepackage{graphicx}
\usepackage{verbatim}

\newtheorem{theorem}{Theorem}[section]

\newtheorem{proposition}[theorem]{Proposition}
\newtheorem{remark}[theorem]{Remark}
\newtheorem{definition}[theorem]{Definition}

\numberwithin{equation}{section}

\usepackage[a4paper,margin=3cm]{geometry}
\newcommand{\Om}{\Omega}
\newcommand{\ep}{\varepsilon}

\newcommand{\R}{\mathbb{R}}
\newcommand{\N}{\mathbb{N}}

\newcommand{\Hs}{\mathcal{H}}
\newcommand{\B}{\mathcal{B}}
\newcommand{\C}{\mathcal{C}}

\newcommand{\xp}{\left( x \right)}

\newcommand{\sdist}{{\rm{sdist}}}

\newcommand{\cngK}{\xrightarrow{\mathcal{K}}}
\newcommand{\Int}{{\rm{Int}}}

\newcommand{\rk}{ \right\}}
\newcommand{\lk}{ \left\{}

\newcommand{\sym}{{\rm sym}}
\newcommand{\dist}{{\rm dist}}

\newcommand{\kn}{{k_n}}

\newcommand{\AH}{ {\rm AH}(\Om)}

\newcommand{\wt}{\widetilde}

\newcommand\restr[2]{{
  \left.\kern-\nulldelimiterspace 
  #1 
  \littletaller 
  \right|_{#2} 
  }}

\tikzset{every picture/.style={line width=0.75pt}} 

\begin{document}
	\sloppy
	\title[Multilayer films]{Solutions for a free-boundary problem modeling  multilayer films with coherent and incoherent interfaces}

	\author[R. Llerena]{Randy Llerena}
	\address[Randy Llerena]{Research Platform MMM ``Mathematics-Magnetism-Materials" - Fak. Mathematik Univ. Wien, A1090 Vienna}
	\email{randy.llerena@univie.ac.at}
	
	\author[P. Piovano]{Paolo Piovano}
	\address[Paolo Piovano]{Dipartimento di Matematica, Politecnico di Milano, P.zza Leonardo da Vinci 32, 20133 Milano, Italy\footnote{MUR Excellence Department 2023-2027}
 }
	\email{paolo.piovano@polimi.it}
	
	\date{\today} 
	\begin{abstract}
		In this paper, we move forward from the results of \cite{LlP} by introducing a variational model for the study of  multilayer films that allows for the treatment of both coherent and incoherent interfaces between layers. The model is designed in the framework of the theory of Stress Driven Rearrangement Instabilities, which are characterized by the competition between elastic and surface energy effects. The surface of each film layer is assumed to satisfy the ``exterior graph condition'' already introduced in \cite{LlP}, for which bulk cracks are allowed to be of non-graph type. By applying the direct method of calculus of variations under a constraint on the number of connected components of the cracks not connected to the surface of the film layers the existence of energy minimizers is established in dimension 2. As a byproduct of the analysis the state of art  on the variational modeling of single-layered films deposited on a fixed substrate is advanced  by letting the substrate surface free, by addressing the presence of multiple layers of various materials, and by including the possibility of delamination  between the various film layers.

		 \end{abstract}
	
	\subjclass[2010]{49J10, 49Q15, 49Q20, 35R35 , 74A45, 74A50 , 74G65}
	\keywords{free boundary problem, surface energy, delamination, elastic energy, deformable layers, multiphase morphology, multilayer films}
	
	\maketitle
	
	\tableofcontents
	
	\section{Introduction}

In this manuscript, we address the problem of modeling the morphology of multilayered film composites consisting of different crystalline materials deposited on a substrate. 
The goal is to advance the literature on the variational modeling of single-layered films deposited on a fixed substrate \cite{ChB, DP, DP2, FFLM2}   in a twofold direction: on the one hand, by letting the substrate surface free and by addressing the presence of multiple layers of various materials, and, on the other hand, by including into the analysis the possibility of a failure of the film coatings, since,  as described in \cite{SroloviztAnderson} for the case of some oxide films, the compressive stresses generated during film growth can lead to the delamination (and the buckling) between different layers.

Nowadays film-based nanostructures find several applications, in particular for the manufacturing of electronic and photonic devices, such as for the creation of their semiconductor components, and of solar and photovoltaic cells. The great interest that films  and, in particular, multilayer films \cite{SRS,Tersoff2,Tersoff1}, created by vapor deposition of different material constituents,   continue attracting is due to the fact that, as they are self-assembled heterostructures, their employment represents  one of the nanostructure design methods with most  feasibility potential; 
therefore, any advancement in the mathematical modeling of film and multilayer film materials can have an important practical impact for their design control. Examples of multilayer films that are used for optoelectronic applications are multiple quantum well structures with alternating compressive and tensile strained layers, and short-period \emph{quantum-dot superlattices}. Also for the latter, as described in \cite{Tersoff2,Tersoff1}, it is really the superposition of various layers of materials that allows to reach the highest degree order needed for the applications with respect to the size, the density and the distribution of the quantum dots.

The adopted strategy consists in combining the implementation to the multiphase setting of the film models considered in \cite{ChB, DP, DP2, FFLM2}, in which delamination is not taken into account, with the recent results for a two-phase setting of \cite{LlP}, in which the interfaces between phases are instead allowed to present both coherent and incoherent portions. With coherency here we intend a microscopic organization of atoms that can be regarded as a (possibly deformed) uniform lattice that is homogeneous through the interface, while with incoherency we refer to the presence of debonding and delamination at the interface. 
In this way the extension of the single-layer literature to the multilayer setting (with possible delamination at each layer interface) is performed within the theory of stress driven rearrangement instabilities (SDRI) \cite{AT,D,G,S}, which was also at the basis of the variational single-layer models introduced in \cite{spencer1,spencer2} and analytically validated in \cite{ChB,DP,DP2,FFLM2}. In fact, as in \cite{spencer1, spencer2} for thin films, and more generally for free crystals in \cite{KP, KP1,KP2}, we model the mismatch between the free-standing equilibrium lattices of the materials of each pair of film layers and of the first layer with the substrate by means of the so-called \emph{mismatch strain} in the elastic energy. As described by the SDRI theory the lattice mismatch is responsible for the migration of the atoms of each phase from their crystalline order, since the lattice mismatch induces large stresses in the bulk material and, in order to release the related elastic energy, the atoms move forming corrugations, cracks,  and other interface instabilities  \cite{AT,G,D,S}.

In regard to the literature results for settings with phase interfaces exclusively assumed to be coherent, we refer to the literature on the optimal shape of partitions in the absence of elastic effects, which was initiated by  Almgren in \cite{Algrem}, who formulated the problem in $\mathbb{R}^d$, for $d>1$,  for surface tensions proportional at each interface. By working in the framework of \emph{integral currents} of geometric measure theory he singled out a condition referred to as ``partitioning regularity", that ensures the lower semicontinuity of the overall surface energy with respect to the $L^1$-convergence of the sets in the partition. 
Then, Ambrosio and Braides expanded the scope in \cite{AB,AB2} by including also non-proportional surface tensions and by introducing an integral condition called $BV$-\emph{ellipticity}, which they proved to be both sufficient and necessary for the $L^1$-lower semicontinuity. Afterwards, various other conditions have been introduced and studied, such as $B$-\emph{convexity} and \emph{joint convexity}, in the attempt of finding a more practical condition than $BV$-\emph{ellipticity}, as the latter can be challenging to be verified as it represents the analogous of Morrey’s \emph{quasi-convexity} condition in the setting of \emph{Caccioppoli partitions}. $BV$-ellipticity though remains the only known condition characterizing the the $L^1$-lower semicontinuity apart from specific contexts (we refer to \cite{cara2, Morgan} for more details), and the fact that it coincides with the \emph{triangle-inequality condition}, which is simpler to check, for the case with 3 phases \cite{AB,AB2}. 
Finally, in \cite{friesolo} the analogous version of the  $BV$-ellipticity condition in the framework of $BD$-\emph{spaces} has been studied.

Instead, in regard to the settings with only incoherent interfaces, we refer to the results obtained with respect to the related  \emph{Mumford-Shah} problem for also the application to \emph{image segmentation}, which was actually originally introduced in  \cite{MS} as a  multiphase formulation. 
 In this context, interfaces represent the contours of the image color areas that can be characterized as the discontinuity set of an auxiliary state function. We refer to \cite{AFP,DMS1} for existence and \emph{Ahlfors-type regularity} results in the context of a single phase, which has been then extended also to the \emph{Griffith model} in fracture 
mechanics in the context of linear elasticity with respect to vectorial state functions representing the bulk displacement of crystalline materials  \cite{CC1,FrMa}. Finally, Bucur, Fragal\`a, and Giacomini addressed the original multiphase setting of  \cite{MS} in \cite{BFG} and \cite{BFG2} by providing a rigorous mathematical formulation with incoherent interfaces (see also \cite{CTV} for a related multiphase boundary problem for reaction-diffusion systems).

In \cite{BFG} they recover Ahlfors-type regularity results for an  \emph{ad hoc} nonstandard notion of multiphase local almost-quasi minimizers for an energy accounting for the incoherent portions of each interface and disregarding the contribution of the remaining coherent portions. 
Afterwards, in \cite{BFG2} the same Authors introduced what they refer to as the \emph{multiphase {M}umford-{S}hah problem}, that is characterized by the sum of possibly different Mumford-Shah-type energy contributions, each related to a different phase, to which an extra term (justified on statistical reasons) is added. Such extra term is needed as otherwise minimizing configurations would present a single phase. However, in \cite{BFG2} coherent interfaces are not counted in the energy as ``no-jump interface portions'' along the reduced phase boundary are weighted in each phase energy in the same way as the jump portions.

To include in our model the interplay between coherency and incoherency, by allowing each phase interface to present also both coherent and incoherent portions, we adopt the strategy initiated in \cite{LlP} for the setting with a film phase deposited on a substrate. Since the results in \cite{LlP} regards $d=2$ and were achieved under a so-called  \emph{exterior graph constraint} on the substrate surface, in order to implement those results to multiple film phases, we also restrict to $d=2$ and we assume on both the substrate surface and the film profiles the exterior graph constraint. We notice that even in the presence of such condition internal cracks in each film layer and in the substrate are allowed to be also of non-graph type. 

We denote by $\Om:=(-l,l)\times (-L,+\infty)$ for positive parameters $l,L \in \R$ the region where the multilayer film and the substrate are located, and, given $\alpha\in\mathbb{N}$, we denote a multilayered film composite with $\alpha$ layers on top of the substrate phase $S_0$, which is also denoted in the following as the \emph{0th layer}, by $S_\alpha$. Furthermore, for each $j\in\{0,\dots,\alpha\}$ we assume that the profile of each \emph{$j$th layer} is parametrizable by a \emph{height function} $h^j:[-l,l] \mapsto [-L,+\infty)$ measuring the thickness of the profile of \emph{$j$th composite} $S_j$,  i.e., the $j$-layered film composite including all $i$th layers for $i\in\{0,\dots,j\}$, by assuming that $h^{j-1}\leq h^{j}$ for $j\in\{1,\dots,\alpha\}$. We also denote by $K^j \subset \overline{\Int(S_{h^j})}$, where $S_{h^j}$ is  the subgraph of $h^j$, the \emph{cracks of the  $j$th composite}, which are assumed such that $K^j\cap \Int(S_{h^{j-1}})\subset K^{j-1}$, so that then for $j\in\{1,\dots,\alpha\}$ the \emph{$j$-composite} $S_j$ coincides with 
$$
S_{h^j,K^j}:=S_{h^j} \setminus  K^j,
$$
and the $j$th film layer coincides with   $S_{h^j,K^j} \setminus S^{(1)}_{h^{j-1},K^{j-1}}$ (see Figure \ref{mlfig:my_label}). 
It follows that there is no formal distinction in the hypotheses taken on the substrate phase $S_0$ and the one taken on each $j$th film composite $S_{h^j,K^j}$ (apart from the fact of being contained in all of them).

In particular, by writing that $(h^j,K^j)\in \text{AHK}(\Om)$ we assume that each $j$th layer height function $h^j$ is an upper semicontinuous function with bounded pointwise variation and each $j$th composite crack set $K^j$ is a closed $\mathcal{H}^1$-rectifiable set with finite $\mathcal{H}^1$ measure. 
More precisely, we denote the family $\B^\alpha$ of admissible multilayered film composites in $\Omega$  with $\alpha$ layers (on the  substrate layer), as a $(\alpha+1)$-tuple of all the $j$th composites $S_{h^j,K^j}$ for $j\in\{0,\dots,\alpha\}$, namely
 \begin{align*}
		\B^\alpha:= \{ (S_{h^\alpha,K^\alpha}, \ldots, S_{h^0,K^0}) :   &\text{ $(h^j,K^j)\in \text{AHK}(\Om)$, $h^{j-1}\leq h^{j}$,  $K^j \subset \overline{\Int(S_{h^j})}$,}\\
  &\quad\,\,\text{ $K^j\cap \Int(S_{h^{j-1}})\subset K^{j-1}$ for $j\in\{1,\dots,\alpha\}$}  \}.
\end{align*}

Furthermore, by following  the SDRI theory \cite{AT,D,G,spencer1,S} 
and in the analogy with the single-layer film setting \cite{ChB,DP,DP2,FFLM2,LlP},  we define the family of \emph{admissible configurations} $\C^\alpha$ by
$$
\mathcal{C}^\alpha  := \{(S_{h^\alpha,K^\alpha}, \ldots, S_{h^0,K^0},  u): (S_{h^\alpha,K^\alpha}, \ldots, S_{h^0,K^0}) \in \B^\alpha ,\, u \in \mathrm{H}^1_{\mathrm{loc}} (\Int(S_{h^\alpha,K^\alpha}))\},
$$
where the functions $u$ represent the \emph{bulk displacements} in the multilayered film composites, and we consider a  configurational  energy $\mathcal{F}_\sigma^\alpha: \C \to [-\infty, \infty]$ given by the sum of an elastic energy $\mathcal{W}$ and a surface energy $\mathcal{S}_\sigma^{\alpha}$, namely,
$$
	\mathcal{F}_\sigma^\alpha (S_{h^\alpha,K^\alpha}, \ldots, S_{h^0,K^0},  u) : = \mathcal{W}({S_{h^\alpha,K^\alpha}, \ldots, S_{h^0,K^0},}
 u)  + \mathcal{S}^\alpha_\sigma (S_{h^\alpha,K^\alpha}, \ldots, S_{h^0,K^0})
$$
for any $(S_{h^\alpha,K^\alpha}, \ldots, S_{h^0,K^0},  u)\in \C^\alpha$. The parameter $\sigma\in\{s,t\}$ is introduced to differentiate the setting in which each layer is assumed to interact with all the other $\alpha$ layers of the composite for $\sigma=t$ from the setting in which each $j$th layer interacts only with the two layers surrounding it (if present), i.e.,   with the  $(j+1)$th layer located directly above and the  $(j-1)$th layer directly underneath, for $\sigma=s$. Let us  refer in the following to $\mathcal{F}_\sigma^\alpha$ (and to $\mathcal{S}_\sigma^\alpha$) as the \emph{$\alpha$-layered total and sequential (surface) energy} for $\sigma=t$ and $\sigma=s$, respectively.
 
The elastic energy $\mathcal{W}$ is defined in $\C^\alpha$ by 
$$
	 \mathcal{W}({S_{h^\alpha,K^\alpha}, \ldots, S_{h^0,K^0},}
 u) : = \int_{S_{h^\alpha, K^\alpha}} W(x, Eu(x) - {E_0^\alpha} (x)) \, dx,
$$
where the elastic density $W$ denotes the quadratic form
	\begin{equation*}
		W \left( x, M \right) := \mathbb{C}\xp M : M,
	\end{equation*}
	defined for the fourth-order tensor $\mathbb{C}: \Om \to \mathbb{M}^2_\sym$, $E$ denotes the symmetric part of the gradient, i.e., $E(v) := \frac{\nabla v + \nabla^T v}{2}$ for any $v \in H^1_\mathrm{{loc}}(\Int (A); \R^2)$ for a set $A$, and represents the \emph{strain}, and ${E_0^\alpha}$ is
	the mismatch strain $x \in \Om \mapsto {E_0^\alpha} \xp \in \mathbb{M}^{2}_\mathrm{{sym}} $ defined as
	 \begin{equation*}
		{ E_0^\alpha} := \begin{cases}
			E(u_0^\alpha) & \text{in } \Om \setminus S_{h^{\alpha-1}}, \\
            E(u_0^i) & \text{in } \Int(S_{h^i}) \setminus S_{h^{1-1}} \text{ for $i = 1 , \ldots, \alpha-1$}  \\
            0 & \text{in } \Int (S_{h^0,K^0}),
		\end{cases} 
	\end{equation*}
	with respect to  fixed $\alpha$ functions  $u_0^i  \in H^1(\Om; \R^2)$ for $i\in\{1,\dots, \alpha\}$. 

 Both the surface energies $\mathcal{S}_\sigma^\alpha$ for $\sigma=s,t$ are given as  sums of  pairwise contributions $\mathcal{S}^{(i,j)} : \B^1 \to [0, \infty]$  for $0\le i < j \le \alpha$ defined by 
\begin{equation}\label{ijcontributions}\mathcal{S}^{(i,j)}(S_{h^j,K^j}, S_{h^i,K^i})  := \int_{\partial S_{h^i,K^i} \cup \partial S_{h^j,K^j}  } {\psi_{i,j} (z, \nu) \, d\Hs^{1}},
\end{equation}
for every admissible multilayered composite $(S_{h^\alpha,K^\alpha}, \ldots, S_{h^0,K^0}) \in \B^\alpha $, where $\psi_{i,j}$ denotes the anisotropic surface tension that takes different definition with respect to the various portions of  $\partial S_{h^i,K^i} \cup \partial S_{h^j,K^j}$. More precisely, in order to properly define $\psi_{i,j}$ we  consider the three surface tensions $\varphi_{i},\, \varphi_{j},\,  \varphi_{ij}: \overline{\Om} \times \R^2\to [0,\infty]$ characterizing the vapor-$i$th layer interface, the vapor-$j$ layer interface, and the $i$th layer-$j$ layer interface.
Moreover, in order to address both the wetting and dewetting regimes with respect to the materials of each pair of film layers, we introduce two additional surface tensions for each pair $(S_{h^j,K^j}, S_{h^i,K^i})$, denoted as the \emph{$i,j$ regime surface tensions}, which are defined as follows:
 $$
\varphi^1_{ij} := \min\{\varphi_i, \varphi_j + \varphi_{ij} \}\qquad\text{and}\qquad \varphi^2_{ij} : = \min\{\varphi_i,\varphi_j\},
$$
in analogy to the definitions given in \cite{LlP} for two-phase setting. 
It follows that
\begin{equation*}
		\psi_{i,j}(x, \nu(x)) :=  \begin{cases} 
  \varphi_{j}(x, \nu_{S_{h^j,K^j}} (x))  & x \in \Om \cap (\partial^* S_{h^j,K^j} \setminus \partial^* S_{h^i,K^i})
  \\
		\varphi^1_{ij} (x, \nu_{S_{h^j,K^j}}(x)) & x \in \Om \cap   \partial ^*S_{h^i,K^i} \cap \partial^* S_{h^j,K^j}, \\
  \varphi_{ij} (x, \nu_{S_{h^i,K^i}} \xp ) & x \in\Om \cap (\partial ^*S_{h^i,K^i} \setminus \partial S_{h^j,K^j}) \\
  (\varphi_j + \varphi^1_{ij}  )(x, \nu_{S_{h^j,K^j}} (x) ) & x \in\Om \cap  \partial ^*S_{h^i,K^i} \cap \partial S_{h^j,K^j}  \cap S_{h^j,K^j}^{(1)},\\
   2\varphi_j(x, \nu_{S_{h^j,K^j}} \xp ) 
   & x \in \Om \cap \partial S_{h^j,K^j} 
   \cap  S_{h^j,K^j}^{(1)} \cap  S_{h^i,K^i}^{(0)},  \\
   2 \varphi^2_{ij}(x, \nu_{S_{h^j,K^j}} \xp )
   & x \in  \Om \cap \partial {S_{h^j,K^j}} 
   \cap {S_{h^j,K^j}}^{(0)},  \\
   2\varphi_{ij} (x, \nu_{S_{h^i,K^i}} \xp)
    & \begin{aligned}
        x \in  \Om &\cap (\partial S_{h^i,K^i} \setminus \partial S_{h^j,K^j}) \\
        &\cap \left(S_{h^i,K^i}^{(1)} \cup S_{h^i,K^i}^{(0)}\right)  \cap S_{h^j,K^j}^{(1)},
    \end{aligned}   \\
     \varphi^1_{ij} (x, \nu_{S_{h^j,K^j}}(x)) 
     & x \in  \Om \cap \partial S_{h^i,K^i} \cap \partial S_{h^j,K^j} \cap S_{h^i,K^i}^{(1)},  
 \end{cases} 
	\end{equation*}
 where, given a set $U\subset\mathbb{R}^2$, $\nu_U$, $\partial^*U$, and $U^{(\alpha)}$  denote, when well defined, the outward pointing normal to $\partial U$, the \emph{reduced boundary}, and the set of points of density $\alpha\in[0,1]$, respectively. Notice that if $\alpha =1$ the energy $\mathcal{S}^{(0,1)}$ coincides with the surface energy defined in \cite{LlP} as, by following the notation of \cite{LlP} we have that $\varphi_0 := \varphi_{\mathrm{S}}, \varphi_1  := \varphi_{\mathrm{F}}, \varphi_{01}  := \varphi_{\mathrm{FS}}$ and as a consequence $ \varphi^1_{01} = \varphi $ and $\varphi^2_{0,1} = \varphi'$. More precisely, the \emph{$\alpha$-layered surface energies}  $\mathcal{S}_\sigma^\alpha : \B^\alpha \to [- \infty, \infty]$ are given for $\sigma=s,t$ by 
$$
\mathcal{S}^{\alpha}_s(S_{h^\alpha,K^\alpha}, \ldots, S_{h^0,K^0}) :=\sum_{j=1}^{\alpha}\mathcal{S}^{(j-1,j)}(S_{h^j,K^j}, S_{h^{j-1},K^{j-1}}), 
$$
and
$$
\mathcal{S}_t^{\alpha}(S_{h^\alpha,K^\alpha}, \ldots, S_{h^0,K^0}) :=\sum_{j=1}^{\alpha}\sum_{i=0}^{j-1} \mathcal{S}^{(i,j)}(S_{h^j,K^j}, S_{h^i,K^i}),
$$
respectively. Notice that we also address the more general case in which the surface density $\psi_{i,j}$ in \eqref{ijcontributions}  are not interconnected  for different pairs of indexes  $0\le i < j \le \alpha$  
(see Remarks \ref{moregeneral_settings} and
\ref{moregeneral_settings_results} for more details).

It was observed in  \cite{KP, LlP} that the family $\B^1$ lacks compactness with respect to the signed distance convergence. In order to overcome this issue and being able to apply  Go{\l}\k ab's’s Theorem \cite{G} to recover compactness, we impose a constraint $m_j\in\N$ on the number of connected components of the cracks of the $j$th composite that are not connected to the $j$th layer for $j=0,\dots,\alpha$. Therefore, we restrict to the family of configurations $\C^\alpha_{\bm{m}} \subset \C^\alpha$ for which such constraints hold, where $\textbf{m}:=(m_0, \ldots, m_\alpha)\in \N^{\alpha+1}$.


\begin{figure}[ht]
     \centering
     \includegraphics{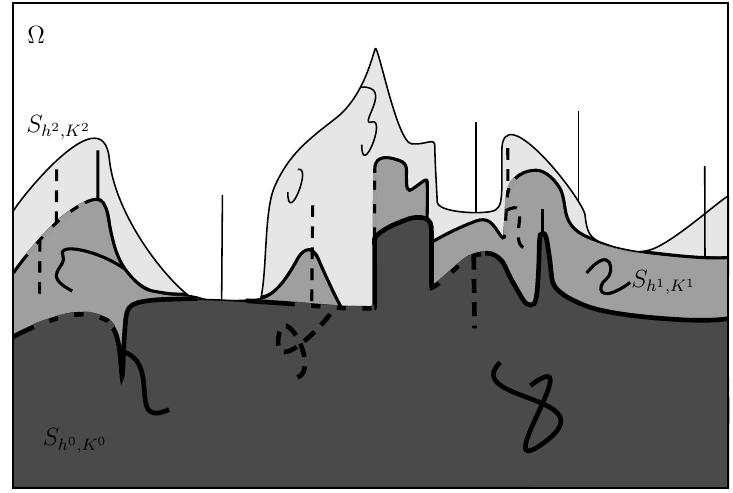}
     \caption{\small A multilayered film composite with 2 layers (on the substrate 0th layer $S_{h^0,K^0}$) associated to an admissible configuration $(S_{h^2,K^2}, S_{h^1,K^1}, S_{h^0,K^0},u)\in\C^2_\textbf{m}$   (see Definition \ref{def:multilayers}) is represented by indicating each $j$th  layer with a gray color with decreasing value with respect to the increasing order of the index $j=0,1,2$. Furthermore, the $j$th layer is indicated with a thinner line with respect to the increasing order of the index $j=0,1,2$, and for the $0$th and 1st layer we distinguish between their coherent and incoherent portions by using a dashed or a continuous line, respectively. }
     \label{mlfig:my_label}
 \end{figure}

 The main goal of the paper is to prove that, given   $\alpha\in\mathbb{N}$, $\sigma\in\{s,t\}$, and a  family of area constraints $\{ \mathbbm{v}_j\}_{i=0}^{\alpha} \subset  [\mathcal{L}^2(\Om)/2, \mathcal{L}^2(\Om)]$ for each  $j$th composite, the minimum problem 
 \begin{equation}
				\label{mleq:constintr}
	\inf_{\tiny  \begin{aligned}
	&(S_{h^\alpha,K^\alpha}, \ldots, S_{h^0,K^0},u) \in \mathcal{C}^\alpha_\mathbf{m} , \\ 
	& \vspace{-0.01cm} \mathcal{L}^2({S_{h^j,K^j}})= \mathbbm{v}_j,\, \text{for }j = 0, \ldots, \alpha 
	\end{aligned} }{\mathcal{F}_\sigma^\alpha (S_{h^\alpha,K^\alpha}, \ldots, S_{h^0,K^0},u)} .
			\end{equation}
   admits a solution. 
   
 To do that we employ the \emph{Direct Method} of Calculus of Variations that consists in finding a proper topology $\tau_{\mathcal{C}^\alpha}$ weak enough to prove compactness in $\mathcal{C}^\alpha_\mathbf{m} \subset \mathcal{C}^\alpha $ and strong enough to have lower semicontinuity of $\mathcal{F}^\alpha$ in $\mathcal{C}^\alpha_\mathbf{m}$. The  topology  $\tau_{\mathcal{C}^\alpha}$ that we consider is the one for which  the convergence
 $$(S_{h^\alpha_k,K^\alpha_k}, \ldots, S_{h^0_k,K^0_k},u_k) 
 \xrightarrow[k\to\infty]{\tau_{\C^\alpha}}
(S_{h^\alpha,K^\alpha}, \ldots, S_{h^0,K^0},u)$$
is equivalent to 
   \begin{align*}
&\begin{cases}
\text{for every $i = 0 , \ldots , \alpha $, $\sup_{k\in\mathbb{N}}{ \Hs^1 \left( \partial S_{h^i_k,K^i_k} \right)} < \infty$,}\\
\text{$\mathrm{ sdist} \left( \cdot, \partial S_{h^i_k,K^i_k} \right) \xrightarrow[k\to\infty]{} \mathrm{ sdist} \left( \cdot, \partial S_{h^i,K^i} \right)$ locally uniformly in $\R^2$ and} \\
u_k \xrightarrow[k\to\infty]{} u \quad\text{a.e.\ in $\mathrm{{Int} \left(S_{h^{\alpha},K^\alpha}\right)}$,}
\end{cases}
\end{align*}
 where the \emph{signed distance function} is defined for any $E \subset \R^2$ as follows
	$$
	    \sdist(x, \partial E) := \begin{cases}
	        \dist(x, E) & \text{if } x \in \R^2 \setminus E, \\
	        -\dist(x, E) & \text{if } x \in  E.
	    \end{cases}
	$$

   For the compactness we implement in the multilayer setting the compactness results proven for the substrate in \cite{LlP}, which were based on \cite{ChB,DP,DP2, FFLM2} (with the difference that instead of lower semicontinuous graph we assume an upper semicontinuity property). 
   We notice that in order to include incoherency in the setting of \cite{ChB,DP,DP2, FFLM2} we implement for multilayers the setting of \cite{LlP}, where in the elements in the $(\alpha+1)$-tuple are not each film layers, but the $j$th composites. In particular, this allows to include in the model also the possible presence of a countable island of one material onto the other layers. In order to establish the lower semicontinuity property we instead proceed by induction by directly using the lower semicontinuity result of \cite{LlP} for the basis of the induction. 

   We conclude by describing the organization of the paper. In Section \ref{notation} we state the notation and recall fundamental definitions used throughout the paper. In Section \ref{mathset} we introduce the model and the main results of the paper. 
   In Section \ref{singlelayer} we prove the existence of minimizers for single-layer films with delamination. Finally, in Section \ref{multisection} we prove the existence result for the minimum problem \eqref{mleq:constintr} with a finite number $\alpha$ of layers over the substrate 0th layer. 
	

 \section{Notation}
 \label{notation}
	In this section, we collect the relevant notation used throughout the paper.
	\subsection*{Linear algebra}
 	We consider the orthonormal basis $\left\{ \mathbf{e}_\mathbf{1}, \mathbf{e}_\mathbf{2} \right\} =  \left\{ (1,0), (0,1) \right\}$ in $\R^2$ and indicate the coordinates of points $x$ in $\R^2$ by $(x_1,x_2)$. We indicate by $a \cdot b:=\sum_{i=1}^2 a_i b_i$ the Euclidean scalar product between points $a$ and $b$ in $\R^2$, and we denote the corresponding norm by $|a|:=\sqrt{a \cdot a}$. 
	
	Let $\mathbb M^2$ be the set of $(2 \times 2)$-matrices and by $\mathbb M^{2}_{\rm sym}$ the space of symmetric $(2 \times 2)$-matrices. The space $\mathbb M^2$ is endowed with Frobenius inner product $E:F:= \sum_{i,j =1}^{2} E_{ij} F_{ij}$ and, with a slight abuse of notation, we denote the corresponding norm by $|E|:=\sqrt{E:E}$. 
	
	\subsection*{Topology}
	Since the model considered in this manuscript is two-dimensional, if not otherwise stated, all the sets are contained in $\R^2$. 
    For any set $E \subset \R^2$, we denote by $E^c$ the complement of $E$ and 
by $\Int(E)$, $\overline{E}$ and $\partial E$  interior, the closure and the topological boundary of $E$, respectively. 

	Finally, let 
	$\dist(\cdot, E)$ and $\sdist(\cdot, \partial E)$ be the \emph{distance} function from $E$ and the \emph{signed distance} from $\partial E$ respectively, where we recall that $\sdist(\cdot, \partial E)$ is defined by 
	\begin{equation*}
		\sdist(x, \partial E):= \begin{cases}
			\dist (x, E) & \textrm{if } x \in \R^2 \setminus E, \\
			-\dist (x, E) & \textrm{if } x \in E
		\end{cases}
	\end{equation*}
	for every $x \in \R^2$.
	\subsection*{Geometric Measure theory}
	We denote by $\mathcal{L}^2{(B)}$ the $2$-dimensional Lebesgue measure of any Lebesgue measurable set $B \subset \R^2$ and by $\mathbbm{1}_B$ the characteristic function of $B$. For $\alpha \in \left[0, 1 \right]$ we denote by $B^{(\alpha)}$ the set of points of density $\alpha$ in $\R^2$, i.e.,  
	$$ B^{(\alpha)} := \lk x \in \R^2: \lim_{r \to 0}{ \frac{\mathcal{L}^2({B \cap B_r \xp})}{\mathcal{L}^2({B_r \xp})} } = \alpha \rk. $$ 
	
	We denote the distributional derivative of a function $f \in L^1_{{\rm loc}}(\R^2)$ by $Df$ and define it as the operator $D: C^\infty_c (\R^2) \to \R$ such that
	$$\int_{\R^2}{Df \cdot \varphi} =- \int_{\R^2}{f \cdot  \grad \varphi \, dx}, $$
	for any $\varphi \in C^\infty_0 (\R^2)$.

	We denote with $\Hs^1$ the 1-dimensional Hausdorff measure. 
	We say that $K \subset \R^2$ is $\Hs^1$-rectifiable if $0 <\Hs^1 (K) < +\infty$ and $\theta_* (K,x) = \theta^* (K,x)=1 $ for $\Hs^1$-a.e. $x\in K$, where
	$$ \theta_* (K,x) :=  \liminf_{r \to 0^+}{ \frac{\Hs^1 (K \cap B_r \xp) }{2r} } \quad \textrm{and} \quad \theta^* (K,x):=\limsup_{r \to 0^+}{ \frac{\Hs^1 (K \cap B_r \xp) }{2r} }. $$ 
	We define sets of finite perimeter as in \cite[Definition 3.35]{AFP} and the reduced boundary $\partial ^* E$ of a set $E$ of finite perimeter by
	\begin{equation}
		\label{mlreducedboundary}
		\partial^* E := \lk  x \in \R^2: \exists \nu_E \xp := - \lim_{r \to 0} {\frac{D \mathbbm{1}_{E}(B_r \xp)}{\abs{D \mathbbm{1}_ {E} }(B_r \xp)}  }, \abs{\nu_E \xp } =1 \rk,
	\end{equation}
	where we refer to $\nu_E \xp$ as the measure-theoretical unit normal at $x \in \partial E$.

	For any set $E \subset \R^2$ of finite perimeter, by \cite[Corollary 15.8 and Theorem 16.2]{M} it yields that
	\begin{align}
		\Hs^{n-1}(E^{(1/2)} \setminus \partial^* E) = 0 \quad \text{and} \quad \partial ^* E \subset E^{(1/2)}. \label{mleq:differenceonE_halfreducedboudnary} 
	\end{align}
	Moreover, for any set $E \subset \R^2$ of finite perimeter, we have 
		\begin{equation}
			\label{mleq:decompositionE} 
			\partial E = N \cup \partial^* E \cup (E^{(1)} \cup E^{(0)}) \cap \partial E, 
	\end{equation}
	where $N$ is a $\Hs^1$-negligible (see \cite[Section 16.1]{M}). 
	
	\subsection*{Functions of bounded pointwise variation}
	Given a function $h:[a,b] \to \R$ we denote the pointwise variation of $h$ by
	$$ \textrm{Var}\, h := \sup{ \left\{ \sum_i^n \abs{ h(x_i) - h(x_{i-1})}: P:= \{ x_0, \ldots , x_n\} \ \textrm{is a partition of $[a,b]$}  \right\} } . $$
	
	We say that $h:(a,b) \to \R$ has finite pointwise variation if $\textrm{Var}\, h < \infty.$ We recall that for any function $h$ such that $\textrm{Var}\, h < \infty$, $h$ has at most countable discontinuities and there exists $h(x^\pm):= \lim_{z \to x ^\pm}h(z)$. In the following given a function $h: [a,b] \to \R$ with finite pointwise variation, we define
	$$ h^{-} (x):= \min \{ h (x^{+} ), h(x^{-} )\} = \liminf_{z \to x} h(z) $$
	and 
	$$h^{+} (x):= \max \{ h (x^{+} ), h(x^{-} )\} = \limsup_{z \to x} h(z). $$

In view of \cite[Corollary 2.23]{L} and with slightly abuse of notation, the limits
\begin{equation}
    \label{mleq:limitsh}
    h^+(a):= \lim_{x \to a^-} h(x) \quad \text{and} \quad h^-(b):=\lim_{x \to b^-} h(x)
\end{equation}
are finite. 


\section{Mathematical setting and main results}
\label{mathset}
\subsection{Multilayer model}
In this section, we introduce the family of admissible regions with finite number of composite layers and the respective family of admissible configurations. 
	Let $\Om:= (-l,l)\times(-L,{ \infty}) \subset \R^2$ for positive parameters $l,L\in \R$.
	
Analogously to \cite{LlP}, we assume a \emph{graph-crack} constraint of the composite of layers, in other words, we consider a \emph{graph constraint} on the strict epigraph of the composite of layers, while inside of it, we consider \emph{cracks} as closed and  $\Hs^1$-rectifiable sets of $\Om$, roughly speaking, the profile of the composite is given by a function representing its thickness, plus a countable number of external vertical filaments and internal cracks.
More precisely, we consider the \emph{family of admissible heights} $\text{ AH}(\Om)$ 
defined by 
	\begin{align}
		\text{ AH}(\Om) := \left\{  h : [-l,l] \to [0,L]:  h \text{ is upper semicontinuous and } \text{ Var}\, h < \infty 
		\right\} \label{mleq:definitionAH}
		\end{align}
and let $S_h$ denote the closed subgraph  with height $h  \in \text{ AH}(\Om)$, i.e., 
\begin{equation}
	\label{mleq:uppergraphS}
	S_h:= \{ (x,y): -l<x<l, y\le h(x) \}.
\end{equation}
Furthermore, we define the \emph{family of admissible cracks} $\text{AK}(\Om)$  by
\begin{equation}
    \label{mleq:definitioncracksS}
   \text{ AK} (\Omega) := \{ K \subset \Om : \text{ $K$ is a closed set in $\R^2$, 
    $\Hs^1$-rectifiable and } \Hs^1(K) < \infty \}
\end{equation}
and the \emph{family of pairs of admissible heights and  cracks} $\text{AHK}(\Om)$ by
\begin{equation}
    \label{mleq:definitionAHK}
   \text{ AHK} (\Omega) := \{ (h,K)\in \text{ AH}(\Om)\times\text{ AK}(\Om): \,  K \subset \overline{\Int(S_h)} \}.
\end{equation}
Finally, given $(h,K)\in \text{AHK}(\Om)$ we refer to the region characterized as the subgraph  of the height function $h$ without the internal cracks of $K$, namely, 
\begin{equation}
	\label{mleq:substrateheight}
	S_{h,K}  := ({S_h} \setminus K) 
	\cap \Omega,
\end{equation}
as the \emph{(generalized) subgraph with height $h$ and cracks $K$},  
and  we define the \emph{family of admissible subgraphs} as 
\begin{equation}
	\label{mleq:definitionAS}
	\text{ AS}(\Om):= \{S \subset \Om\,:\,  \text{$S=S_{h,K}$ for a pair $(h,K) \in \text{ AHK} (\Omega)$}\}.
\end{equation}

We observe that for every $(h,K) \in \text{AHK} (\Omega)$
\begin{equation}
	\label{mleq:uniongraphs}
	\overline{S_{h,K}} = S_h, \quad \Int(S_{h,K}) = \Int(S_h) \setminus K \quad  \text{and}\quad \partial {S_{h,K}} =\partial {S_h} \cup K .
 \end{equation}
We have that 
$\partial {S_h}$ 
is connected and, $\partial S_h$ and $\partial S_{h,K}$ have finite $\Hs^1$-measure. By \cite[Lemma 3.12 and Lemma 3.13]{F}, for any $h \in \text{AH}(\Om)$, $\partial S_h$ is rectifiable and applying the Besicovitch-Marstrand-Mattila Theorem (see \cite[Theorem 2.63]{AFP}), $\partial {S_h}$ is $\Hs^1$-rectifiable, and hence, $\partial S_{h,K}$ is $\Hs^1$-rectifiable. Furthermore, by applying \cite[Proposition A.1]{KP1} $S_h$ and $S_{h,K}$ are sets of finite perimeter.


\begin{definition}[Admissible multilayers and admissible configurations] 
\label{def:multilayers}
\normalfont
	We define the family of two layers $\B^1$ by
	\begin{multline*}
\B^1 : = \{(S_{h^1,K^1}, S_{h^0,K^0}) : \text{ for $i =0,1$ there exists } (h^i,K^i) \in \mathrm{AHK}(\Om), \,  S_{h^i,K^i} \in  \mathrm{AS}(\Om), \\
 h^0 \le h^1 \, \text{and } \partial S_{h^1,K^1} \cap \Int(S_{h^0,K^0}) = \emptyset  \} \subset  \mathrm{AS}(\Om) \times  \mathrm{AS}(\Om). 
\end{multline*}
Let $\alpha \in \N$, we define the family of admissible $(\alpha+1)$-layers $\B^\alpha$ by
\begin{multline*}
		\B^\alpha:= \{ (S_{h^\alpha,K^\alpha}, \ldots, S_{h^0,K^0})\in  [  \mathrm{AS}(\Om)]^{\alpha+1} : \,    (S_{h^i,K^i}, S_{h^{ i-1},K^{i-1}} ) \in \B^1 \\   \text{ for every $1 \le i 
  \le \alpha$}  \},
\end{multline*}
where $[ \mathrm{AS}(\Om)]^{\alpha+1} : =  \mathrm{AS}(\Om) \times \ldots \times  \mathrm{AS}(\Om)$ represents the $(\alpha+1)$ cartesian product of $\mathrm{AS}(\Om)$. 

We define the family of admissible configurations by
$$
\mathcal{C}^\alpha  := \{(S_{h^\alpha,K^\alpha}, \ldots, S_{h^0,K^0},  u): (S_{h^\alpha,K^\alpha}, \ldots, S_{h^0,K^0}) \in \B^\alpha ,\, u \in \mathrm{H}^1_{\mathrm{loc}} (\Int(S_{h^\alpha,K^\alpha}))\}.
$$
\end{definition}

\begin{remark} 
In view of Definition \ref{def:multilayers}, it follows that every $(S_{h^\alpha,K^\alpha}, \ldots, S_{h^0,K^0}) \in \B^\alpha$ satisfy the following properties:
\begin{itemize}
\item[(i)] $h^i \le h^j$ and $\partial S_{h^j,K^j} \cap \Int(S_{h^i,K^i}) = \emptyset  $ for every $0 \le i< j \le \alpha $ and hence,  $(S_{h^j,K^j},S_{h^{i},K^{i}} ) \in \B^1$ for every $0 \le i< j \le \alpha$;
    \item[(ii)]  $K^i \subset \overline{\Int (S_{h^i})} \setminus \Int(S_{h^{i-1},K^{i-1}})$
for every  $i= 1, \ldots, \alpha$;
\item[(iii)] $K^i \cap K^{i-2} \cap \Int(S_{h^{i-2}}) \subset K^i \cap \Int(S_{h^{i-1}}) \subset \partial S_{h^i, K^i} \cap \Int(S_{h^{i-1}}) \subset K^{i-1}
$ for every  $i= 2, \ldots, \alpha$.
\end{itemize}

\end{remark}

For any $\alpha \in \N$, motivated in \cite{KP, KP1, LlP} we introduce a notion of convergence for the families $\B^\alpha$ and $\C^\alpha$. 
\begin{definition}[$\tau_{\B^\alpha}$-Convergence] \normalfont
\label{mldef:taubconvergencelambda}
	A sequence $ \{ (S_{h^\alpha_n,K^\alpha_n}, \ldots, S_{h^0_n,K^0_n}) \} \subset \B^\alpha$  ${\tau_\B^\alpha}$-converges to $(S_{h^\alpha,K^\alpha}, \ldots, S_{h^0,K^0})
 \in \B^\alpha$ if 
	\begin{itemize}
		\item[-]  $\sup_{k\in\mathbb{N}}
		\Hs^1 (\partial S_{h^i_k, K^i_k})<\infty$, for every $i = 0, \ldots, \lambda$ 
		\item[-] $\text{sdist} \left( \cdot, \partial S_{h^i_k, K^i_k} \right) \rightarrow \text{sdist} \left( \cdot, \partial S_{h^i,K^i} \right)$ locally uniformly in $\R^2$ as $k \rightarrow \infty$ for every $i = 0, \ldots, \lambda$,
	\end{itemize}
	where 
	$$S_{h^i_k, K^i_k}  :=S_{h^i_k} \setminus K^i_k \quad \text{and} \quad S_{h^i, K^i}  := S_{h^i} \setminus K^i
	,$$
	for every $k \in \N$ and $i = 0, \ldots, \alpha$.
\end{definition}

		\begin{definition}[$\tau_{\mathcal{C}^{\alpha}}$-Convergence] \normalfont
A sequence $ \{ (S_{h^\alpha_n,K^\alpha_n}, \ldots, S_{h^0_n,K^0_n}, u_n) \} \subset \mathcal{C}^\alpha$  ${\tau_{\mathcal{C}^\alpha}}$-converges to $(S_{h^\alpha,K^\alpha}, \ldots, S_{h^0,K^0}, u) \in \mathcal{C}^\alpha$ if 
	\begin{itemize}
		\item [-] $(S_{h^\alpha_n,K^\alpha_n}, \ldots, S_{h^0_n,K^0_n}) \xrightarrow{\tau_{\mathcal{B}^\alpha}} (S_{h^\alpha,K^\alpha}, \ldots, S_{h^0,K^0})$,
		\item[-] $u_n \rightarrow u$ a.e. in ${\Int(S_{h^\alpha,K^\alpha})}$.
	\end{itemize}
\end{definition}

Analogously to \cite{KP,KP1,LlP}, we introduce a subfamily of $\B^\alpha$ subject to a restriction on the number of connected components of the boundary of each composite layer and the subfamily of $\C^\alpha$ with the corresponding configurations.

\begin{definition}\normalfont
\label{def:mconnected}
Let $\alpha \in \N$ and let $\bm{m}:= (m_0, \ldots, m_\alpha) \in \N^{\alpha+1}$.   We refer to 
\begin{multline*}
    \B^\alpha_{\bm{m}} : = \{ (S_{h^\alpha,K^\alpha}, \ldots, S_{h^0,K^0}) \in \B^\alpha : \, \text{$\partial S_{h^i,K^i}$ has at most $m_i$-connected components} \\ \text{for $i = 0 , \ldots, \alpha$}  \}
\end{multline*}
as the families of \emph{admissible multilayers} and to 
$$
\C^\alpha_{\bm{m}}:= \{ (S_{h^\alpha,K^\alpha}, \ldots, S_{h^0,K^0}, u) \in \C: \, (S_{h^\alpha,K^\alpha}, \ldots, S_{h^0,K^0}) \in \B^\alpha_{\bm{m}} \}
$$
as the family of \emph{admissible  configurations}. 
\end{definition}

In the sequel we fix $\alpha \in \N$. Motivated in the model introduced in \cite{LlP}, we consider the surface tension between two layers $\mathcal{S}^{(i,j)}: \mathcal{B}^\alpha \to [0, +\infty]$ in the family of admissible layers $\B^\alpha$ by
$$
\mathcal{S}^{(i,j)}(S_{h^j,K^j}, S_{h^i,K^i})  := \int_{\partial S_{h^i,K^i} \cup \partial S_{h^j,K^j}  } {\psi_{i,j} (z, \nu) \, d\Hs^{1}},
$$
where $(S_{h^\alpha,K^\alpha},\ldots,  S_{h^0,K^0}) \in \B^\alpha$ for $0 \le i< j \le  \alpha$, and the surface tension $\psi_{i,j}$ is defined in different portions of $\partial S_{h^i,K^i} \cup \partial S_{h^j,K^j}$, more precisely, 
 \begin{equation}
		\label{mlpsiij}
		\psi_{i,j}(x, \nu(x)) :=  \begin{cases} 
  \varphi_{j}(x, \nu_{S_{h^j,K^j}} (x))  & x \in \Om \cap (\partial^* S_{h^j,K^j} \setminus \partial^* S_{h^i,K^i})
  \\
		\varphi^1_{ij} (x, \nu_{S_{h^j,K^j}}(x)) & x \in \Om \cap   \partial ^*S_{h^i,K^i} \cap \partial^* S_{h^j,K^j}, \\
  \varphi_{ij} (x, \nu_{S_{h^i,K^i}} \xp ) & x \in\Om \cap (\partial ^*S_{h^i,K^i} \setminus \partial S_{h^j,K^j}) \\
  (\varphi_j + \varphi^1_{ij}  )(x, \nu_{S_{h^j,K^j}} (x) ) & x \in\Om \cap  \partial ^*S_{h^i,K^i} \cap \partial S_{h^j,K^j}  \cap S_{h^j,K^j}^{(1)},\\
   2\varphi_j(x, \nu_{S_{h^j,K^j}} \xp ) 
   & x \in \Om \cap \partial S_{h^j,K^j} 
   \cap  S_{h^j,K^j}^{(1)} \cap  S_{h^i,K^i}^{(0)},  \\
   2 \varphi^2_{ij}(x, \nu_{S_{h^j,K^j}} \xp )
   & x \in  \Om \cap \partial {S_{h^j,K^j}} 
   \cap {S_{h^j,K^j}}^{(0)},  \\
   2\varphi_{ij} (x, \nu_{S_{h^i,K^i}} \xp)
    & \begin{aligned}
        x \in  \Om &\cap (\partial S_{h^i,K^i} \setminus \partial S_{h^j,K^j}) \\
        &\cap \left(S_{h^i,K^i}^{(1)} \cup S_{h^i,K^i}^{(0)}\right)  \cap S_{h^j,K^j}^{(1)},
    \end{aligned}   \\
     \varphi^1_{ij} (x, \nu_{S_{h^j,K^j}}(x)) 
     & x \in  \Om \cap \partial S_{h^i,K^i} \cap \partial S_{h^j,K^j} \cap S_{h^i,K^i}^{(1)},  
 \end{cases} 
	\end{equation}
where $\varphi_j,\, \varphi_{ij}: \overline{\Om} \times \R^2\to [0,\infty]$ and, given also the function $\varphi_{i}: \overline{\Om} \times \R^2\to [0,\infty]$, we define the functions $\varphi^1_{ij}$ and $\varphi^2_{ij}$ in 
 $C ( \overline{\Om} \times \R^2;[0,\infty])$ by 
 $$\varphi^1_{ij} := \min\{\varphi_{i}, \varphi_{j} + \varphi_{ij} \}\qquad\text{and}\qquad \varphi^2_{ij} : = \min\{\varphi_{j}, \varphi_{i}\}.$$
 In view of \cite{LlP}, for every $0 \le i \le j \le \alpha$, 
$\varphi_j, \varphi_i, \varphi_{ij}$ represent  the anisotropic surface tensions of the film/vapor, the substrate/vapor and the substrate/film interfaces, respectively, while $\varphi^1_{ij}$ and $\varphi^2_{ij}$ are referred to as the anisotropic \emph{regime surface tensions} and are introduced to include into the analysis the wetting and dewetting regimes.

\begin{remark}\label{alpha1}
If $\alpha = 1$, we can observe that the surface tension $\mathcal{S}$ considered in \cite{LlP} coincides with $\mathcal{S}^{(0,1)}$ by considering, with respect to the notation of \cite{LlP}, $\varphi_0 := \varphi_{\mathrm{S}}, \varphi_1  := \varphi_{\mathrm{F}}, \varphi_{01}  := \varphi_{\mathrm{FS}}$ and as a consequence $ \varphi^1_{01} = \varphi $ and $\varphi^2_{01} = \varphi'$.
\end{remark}

Now, we are in the position to define  the  \emph{$\alpha$-layered  total surface energy} $\mathcal{S}^\alpha_{t} : \B^\alpha \to [0, \infty]$ and the \emph{$\alpha$-layered sequential surface energy} $\mathcal{S}^\alpha_{s} : \B^\alpha \to [0, \infty]$ by
\begin{equation}\label{tot_surface_energy}
\mathcal{S}^{\alpha}_t(S_{h^\alpha,K^\alpha}, \ldots, S_{h^0,K^0}) :=\sum_{j=1}^{\alpha}\sum_{i=0}^{j-1} \mathcal{S}^{(i,j)}(S_{h^j,K^j}, S_{h^i,K^i}) 
\end{equation}
and
\begin{equation}\label{seq_surface_energy}
\mathcal{S}^{\alpha}_s(S_{h^\alpha,K^\alpha}, \ldots, S_{h^0,K^0}) :=\sum_{j=1}^{\alpha}\mathcal{S}^{(j-1,j)}(S_{h^j,K^j}, S_{h^{j-1},K^{j-1}}), 
\end{equation}
respectively. For simplicity we refer to the energy $\mathcal{S}^\alpha_\sigma : \B^\alpha \to [0, \infty]$  for $\sigma \in \{s,t\}$ as the \emph{$\alpha$-layered surface energy}. 

\begin{remark} \label{moregeneral_settings}
The terms in the sum of both definitions \eqref{tot_surface_energy}  and \eqref{seq_surface_energy} depend on the surface tensions $\psi_{i,j}$ that are defined in terms of the family $\Phi$ of surface tensions given by
$$\Phi:=\{\varphi_j,\varphi_{i,j}:\overline{\Om} \times \R^2\to [0,\infty] \, :\, j=1,\dots,\alpha,i=0,\dots,j-1\}.$$
Furthermore, we observe  that the definition of the  surface tensions $\psi_{i,j}$ and $\psi_{j,k}$  for $k\in\{1,\dots,\alpha\}$, $j\in\{1,\dots,k-1\}$, and $i\in\{0,\dots,j-1\}$ are interconnected, as they both depend on $\varphi_{j}$.
We would like to point out that this is just a choice  done for simplicity (and in relation to the physical meaning of the surface tensions of each pair of  $i$th and $j$th composites for $j=1,\dots,\alpha$ and  $i=0,\dots,j-1$, i.e., $\varphi_j$, $\varphi_i$,  and $\varphi_{i,j}$, which follows from \cite{LlP} and Remark \ref{alpha1}). 
We can avoid this interconnection, for example by working in the following settings:
\begin{itemize}
\item[(i)] given $\varphi_{j}:\overline{\Om} \times \R^2\to [0,\infty]$ in the definition of $\psi_{j-1,j}$ by replacing  $\varphi_{j}$ in the definition:
\begin{itemize}
\item[-] of each  $\psi_{j,k}$ (and of the corresponding $\varphi^1_{j,k}$ and $\varphi^2_{j,k}$) for $k\in\{j+1,\dots,\alpha\}$ with a surface tension $\varphi'_{j,k}$ possibly different from $\varphi_{j}$,
\item[-] of each  $\psi_{i,j}$ (and of the corresponding $\varphi^1_{i,j}$ and $\varphi^2_{i,j}$) for $i\in\{0,\dots,j-2\}$ with a surface tension $\varphi'_{i,j}$ possibly different from $\varphi_{j}$,
\end{itemize}
namely by enlarging the definition of $\Phi$  by considering instead
\begin{align}\label{phi'family}
\Phi':=\Phi\cup\{\varphi'_{j,k},\varphi'_{i,j}: & \,\,\overline{\Om} \times \R^2\to [0,\infty] \, :\nonumber\\
& j=1,\dots,\alpha, k=j+1,\dots,\alpha, i=0,\dots,j-2\}; 
\end{align}
\item[(ii)] by defining some  surface tensions in $\phi\in\Phi'$ as a combination of the other surface tensions, namely the surface tensions in $\Phi'\setminus\{\phi\}$. 
\end{itemize}
\end{remark}

Let $\alpha\in \N$ and let $\sigma \in \{s,t\}$, the \emph{$\alpha$-layered total energy} $\mathcal{F}^\alpha_{\sigma} : \mathcal{C}^\alpha \to [- \infty, \infty]$ is defined by 
$$
	\mathcal{F}^{\alpha}_{ \sigma}(S_{h^\alpha,K^\alpha}, \ldots, S_{h^0,K^0},u):= \mathcal{S}^{\alpha}_{ \sigma}(S_{h^\alpha,K^\alpha}, \ldots, S_{h^0,K^0})  + \mathcal{W}({S_{h^\alpha,K^\alpha}, \ldots, S_{h^0,K^0},}
 u)
$$
for any $(S_{h^\alpha,K^\alpha}, \ldots, S_{h^0,K^0},u) \in \C^\alpha $, where $\mathcal{W}$ stands for the elastic energy, more precisely,  
$$
	\mathcal{W}({S_{h^\alpha,K^\alpha}, \ldots, S_{h^0,K^0},}
 u) : = \int_{S_{h^\alpha, K^\alpha}} W(x, Eu(x) - {E_0^\alpha} (x)) \, dx,
$$
and $W$ is determined by the quadratic form 
\begin{equation*}
		W \left( x, M \right) := \mathbb{C}\xp M : M,
	\end{equation*}
	for a fourth-order tensor $\mathbb{C}: \Om \to \mathbb{M}^2_\sym$, $E$ denotes the symmetric gradient, i.e., $E(v) := \frac{\nabla v + \nabla^T v}{2}$ for any $v \in \mathrm{H}^1_\mathrm{{loc}} (\Om)$ and ${ E_0^\alpha}$ is 
	the mismatch strain $x \in \Om \mapsto {E_0^\alpha} \xp \in \mathbb{M}^{2}_\mathrm{{sym}} $ defined as
	\begin{equation*}
		{E_0^\alpha} := \begin{cases}
			E(u_0^\alpha) & \text{in } \Om \setminus S_{h^{\alpha-1}}, \\
            E(u_0^i) & \text{in } \Int(S_{h^i}) \setminus S_{h^{1-1}} \text{ for $i = 1 , \ldots, \alpha-1$}  \\
            0 & \text{in } \Int (S_{h^0,K^0}),
		\end{cases} 
	\end{equation*}
 
	for a fixed sequence  $\{ u_0^i \}_{i=1}^{\alpha-1} \subset H^1(\Om; \R^2)$.


\subsection{Main results} We state here the main results of the paper. 
Let $\alpha \in \N$. We fix $l, L> 0$ and we consider $\Om : =(-l,l) \times(-L,{\infty})$. For every pair of integers $0\le i\le  j \le \alpha$ we consider  $\varphi^1_{ij}: = \min\{\varphi_i, \varphi_{j} + \varphi_{ij} \}$ and $\varphi^2_{ij} : =  \min\{ \varphi_j, \varphi_i \}$ and we assume throughout this manuscript that:
		\begin{itemize}
			\item[(H1)] $\varphi_{j}, \varphi_{ij}, \varphi^1_{ij}, \varphi^2_{ij} \in C ( \overline{\Om} \times \R^2) $ are Finsler norms such that there exists $c_2 \ge c_1 > 0$ such that
			\begin{equation}
	\label{mleq:H1}
c_1 \abs{\xi} \le  \varphi_{j} (x, \xi), \varphi^1_{ij}(x, \xi), \varphi_{ij} (x, \xi) \le c_2 \abs{\xi} \quad \textrm{for every $x \in \overline{\Om}$ and $\xi \in \R^2$},
			\end{equation}
			\item[(H2)] We have 
			\begin{equation}
				\label{mleq:H2}
				\varphi^1_{ij}(x, \xi) \ge \abs{ \varphi_{ij}(x, \xi) -\varphi_{j} (x, \xi)} \quad \textrm{for every $x \in \overline{\Om}$ and $\xi \in \R^2$.}
			\end{equation}
			\item[(H3)] $\mathbb{C} \in L^\infty (\Om; \mathbb{M}^2_\sym)$  and there exists $c_3 > 0$ such that
			\begin{equation}
			\label{mleq:H3}				\mathbb{C} \xp M:M \ge 2c_3 M:M
			\end{equation}
			for every $M \in \mathbb{M}_{\sym}^{2 }.$
		\end{itemize}
We notice that under assumptions (H1)-(H3), the energy $\mathcal{F}^\alpha_\sigma(S_{h^\alpha,K^\alpha}, \ldots, S_{h^0,K^0},u)\in [0, \infty]$ for every $(S_{h^\alpha,K^\alpha}, \ldots, S_{h^0,K^0},u) \in {\mathcal{C}^\alpha}$.

We now state the main theorem of this paper.

\begin{theorem}[Existence of minimizers]
			\label{mlthm:existence}
Fix $\alpha \in \N$, $\mathbf{m}= (m_0, \ldots,m_\alpha) \in \N^{\alpha+1}$, and $\sigma \in \{s,t\}$. Let $\{\mathbbm{v}_i\}_{i=0}^\alpha \subset [\mathcal{L}^2({\Om})/2, \mathcal{L}^2({\Om})]$ be such that $\mathbbm{v}_{i_1} \le \mathbbm{v}_{i_2}$ for every $0\le i_1< i_2 \le \alpha$, and let $\bm{\lambda} : = (\lambda_0, \ldots, \lambda_\alpha) \in \R^{\alpha+1}$ be such that $\lambda_i> 0$ for every $i=0, \ldots, \alpha$. If (H1)-(H3) holds true,  then both the volume constrained minimum problem 
			\begin{equation}
				\label{mleq:const}
	\inf_{\small \begin{aligned}
	&(S_{h^\alpha,K^\alpha}, \ldots, S_{h^0,K^0},u) \in \mathcal{C}^\alpha_{\mathbf{m}}, \\ 
	&\mathcal{L}^2({S_{h^i,K^i}})= \mathbbm{v}_i,\, i = 0, \ldots, \alpha 
	\end{aligned} }{\mathcal{F}^\alpha_{ \sigma} (S_{h^\alpha,K^\alpha}, \ldots, S_{h^0,K^0},u)}
			\end{equation}
			and  the unconstrained minimum problem
			\begin{equation}
				\label{mleq:uncost}
				\inf_{(S_{h^\alpha,K^\alpha}, \ldots, S_{h^0,K^0},u) \in \mathcal{C}^\alpha_{\mathbf{m}}}{\mathcal{F}^{\alpha,\bm{\lambda}}_{\sigma}(S_{h^\alpha,K^\alpha}, \ldots, S_{h^0,K^0},u) \in \mathcal{C}^\alpha_\mathbf{m}},
			\end{equation}
   where  $\mathcal{F}^{\alpha,\bm{\lambda}}_{ \sigma} : \mathcal{C}^\alpha_{\mathbf{m}}\rightarrow \R$ is defined as
			\begin{equation*}
	\mathcal{F}^{\alpha,\bm{\lambda}}_{ \sigma}(S_{h^\alpha,K^\alpha}, \ldots, S_{h^0,K^0},u):= \mathcal{F}^\alpha_{\sigma}(S_{h^\alpha,K^\alpha}, \ldots, S_{h^0,K^0},u)+ \sum_{i= 0}^\alpha \lambda_i \abs{ \mathcal{L}^2 (S_{h^i,K^i})- \mathbbm{v}_i },  
			\end{equation*}
			admit a solution.
		\end{theorem}

  We employ the Direct Method of Calculus of Variations to prove Theorem \ref{mlthm:existence}. In order to apply this method we prove that any energy equi-bounded sequence $\{(S_{h^\alpha,K^\alpha}, \ldots, S_{h^0,K^0},u)\} \subset \mathcal{C}^\alpha_{\mathbf{m}}$ satisfy the following compactness property.

  \begin{theorem}
\label{mlcompactness:Clambda}
    Fix $\alpha \in \N$, $\mathbf{m}= (m_0, \ldots,m_\alpha) \in \N^{\alpha+1}$, and $\sigma \in \{s,t\}$.  Let $\{(S_{h^\alpha_k,K^\alpha_k}, \ldots ,S_{h^0_k,K^0_k}, u_k)\}_{k \in \N} \subset \mathcal{C}^\alpha_{\mathbf{m}}$ be such that
		\begin{equation}
	\label{mleqthm:compactnessClambda}
			\sup_{k\in\mathbb{N}}{\left( \mathcal{F}^\alpha_{\sigma} (S_{h^\alpha_k,K^\alpha_k}, \ldots ,S_{h^0_k,K^0_k}, u_k)  + {\mathcal{L}^2 \left( S_{h^\alpha_k,K^\alpha_k} \right)} \right)} < \infty.
		\end{equation}
	Then, there exist an admissible configuration $(S_{h^\alpha,K^\alpha}, \ldots ,S_{h^0,K^0},u) \in \mathcal{C}^\alpha_{\mathbf{m}}$ of finite energy, a subsequence $\{(S_{h^\alpha_\kn,K^\alpha_\kn}, \ldots ,S_{h^0_\kn,K^0_\kn}, u_{k_n})\}_{n \in \N}$, a sequence $\{(S_{h^\alpha_\kn,\wt K^\alpha_n}, \ldots, S_{h^0_\kn,\wt K^0_n}, u_{k_n})\}_{n \in \N} \subset \mathcal{C}^\alpha_{\mathbf{m}}$ and a sequence $\{b_n\}_{n\in \N}$ of piecewise rigid displacements associated to $S_{h^\alpha_{k_n}, \widetilde{K}^\alpha_n}$ such that 
 $$
 (S_{h^\alpha_\kn,\wt K^\alpha_n}, \ldots, S_{h^0_\kn,\wt K^0_n}, u_{k_n} + b_n) \xrightarrow{\tau_{\mathcal{C}^\alpha}} (S_{h^\alpha,K^\alpha}, \ldots ,S_{h^0,K^0},u) $$
and 
	\begin{equation}
	\label{mlliminfequalliminflayers}
		\liminf_{n \rightarrow \infty}{\mathcal{F}^\alpha_{\sigma}(S_{h^\alpha_\kn,K^\alpha_\kn}, \ldots ,S_{h^0_\kn,K^0_\kn}, u_{k_n})} = \liminf_{n \rightarrow \infty}{\mathcal{F}^\alpha_{\sigma}(S_{h^\alpha_\kn,\wt K^\alpha_n}, \ldots, S_{h^0_\kn,\wt K^0_n}, u_{k_n} + b_n)   }.
		\end{equation}
 \end{theorem}

 Furthermore, we show that $\mathcal{F}^\alpha_{ \sigma}$ is lower semicontinuous in $\mathcal{C}^\alpha_{\mathbf{m}}$ with respect to the topology $\tau_{\C^\alpha}$ for any $\alpha \in \N$ and $\sigma \in \{s,t\}$.

 \begin{theorem}[Lower semicontinuity of $\mathcal{F}^\alpha_{\sigma}$]
	\label{mlthm:lowersemicontinuitymultilayer}
		Fix $\alpha \in \N$, $\mathbf{m}= (m_0, \ldots,m_\alpha) \in \N^{\alpha+1}$, and $\sigma \in \{s,t\}$. Assume (H1)-(H3). If $\{(S_{h^\alpha_k,K^\alpha_k}, \ldots, S_{h^0_k,K^0_k} , u_k)\}_{k \in \N} \subset \mathcal{C}^\alpha_{\mathbf{m}}$ and $(S_{h^\alpha,K^\alpha}, \ldots ,S_{h^0,K^0}, u) \in \mathcal{C}^\alpha_{\mathbf{m}}$ are such that 
  $$
  (S_{h^\alpha_k,K^\alpha_k}, \ldots, S_{h^0_k,K^0_k} , u_k) \xrightarrow{\tau_{\mathcal{C}^\alpha}} (S_{h^\alpha,K^\alpha}, \ldots ,S_{h^0,K^0}, u),
  $$ then
\begin{equation}\label{mleq:thmlower0}
	\mathcal{F}^\alpha_{ \sigma} (S_{h^\alpha,K^\alpha}, \ldots ,S_{h^0,K^0}, u)  \le \liminf_{k \rightarrow \infty}{ \mathcal{F}^\alpha_{ \sigma}(S_{h^\alpha_k,K^\alpha_k}, \ldots, S_{h^0_k,K^0_k} , u_k)}.
		\end{equation}
	\end{theorem}

We conclude this section by addressing the more general setting introduced in  Remark \ref{moregeneral_settings} regarding the family of surface tensions $\Phi'$ defined in \eqref{phi'family}.

 \begin{remark}\label{moregeneral_settings_results} 
We observe that Theorems \ref{mlthm:existence}--\ref{mlthm:lowersemicontinuitymultilayer} 
continue to hold in the more general setting for the surface energies \eqref{tot_surface_energy}  and \eqref{seq_surface_energy} that is described in Remark \ref{moregeneral_settings}-(i) (and so,  also for the setting of  Remark \ref{moregeneral_settings}-(ii))  by simply adapting   hypotheses (H1) and (H2) to the more general definition of $\psi_{i,j}$ given in Remark \ref{moregeneral_settings}-(i) in terms of the family $\Phi'$ of surface tensions  defined in \eqref{phi'family}. More precisely it suffices to  replace $\varphi_j$ in (H1) and (H2) with $\varphi'_{i,j}\in\Phi'$ for all  $j\in\{2,\dots,\alpha\}$, and $i\in\{0,\dots,j-2\}$.  In this way, as the application of   \cite[Theorem 4.2]{LlP},	\cite[Theorem 4.3]{LlP}, and \cite[Theorem 5.13]{LlP}, respectively, in Proposition  \ref{mlcor:compactness}, and in Theorems  \ref{mlthm:compactnesslayers} and  \ref{mlthm:lowersemicontinuitylayer}, is preserved, the same proof of Theorems \ref{mlthm:existence}--\ref{mlthm:lowersemicontinuitymultilayer} yields the corresponding results in  the setting of  Remark \ref{moregeneral_settings}. 
 \end{remark}



\section{Single-layer films with delamination}
\label{singlelayer}
In order to establish Theorem \ref{mlthm:existence} we use an induction argument, and in this section, we prove the basis of the induction. More precisely, we prove Theorem \ref{mlthm:existence} by assuming that $\alpha =1$ and hence, in the following,  we consider $\bm{m} := (m_1, m_0) \in \N^2$. We begin by observing that the double-layered film setting of $\alpha =1$ is a particular case of the two-phase setting considered in  \cite{LlP}, with the only difference that the ``exterior graph condition'' is assumed not only on the substrate region but also on the film phase. The analogy comes also from the fact, that as proved below, for energy equi-bounded admissible configurations in $\mathcal{C}^1$, we can easily reduced to bounded rectangular containers $\wt \Om := (-l,l) \times (-L, \wt L)$
for a properly chosen constant $\wt L>0$. 

We recall that in \cite{LlP} the families  of admissible regions $\mathcal{B}(\wt \Om)$ and $\mathcal{B}_{\mathbf{m}}(\wt \Om)$ are defined as 
\begin{align*}
	{\mathcal{B}}(\wt \Om):=\Large\{ (A,S_{h,K}):   \quad& (h,K) \in \text{AHK}(\wt \Om), \text{ $A$ is $\mathcal{L}^2$-measurable set with $S_{h,K}\subset  \overline{A}    \subset \overline{\wt \Om}$}\\   &\text{such that $\partial A \cap \Int(S_{h,K}) = \emptyset$, $\partial A$ is $\Hs^1$-rectifiable, } \\ &\text{$\Hs^1 (\partial A)+\Hs^1 (\partial S) < \infty$}  \}
\end{align*}
and
\begin{multline*}
    \mathcal{B}_{\mathbf{m}}(\wt \Om): = \{(A,S_{h,K}) \in{\mathcal{B}}(\wt \Om): \, \partial A \text{ and } \partial S_{h,K} \text{ have at most} \\ \text{$m_1$ and $m_0$ connected components, respectively} \}.
\end{multline*}
Notice that $\mathcal{B}^1(\wt \Om)\subset \mathcal{B}(\wt \Om)$ and $\mathcal{B}^1_{\mathbf{m}}(\wt \Om)\subset \mathcal{B}_{\mathbf{m}}(\wt \Om)$.
Furthermore,  the families of admissible configurations in \cite{LlP} are $\mathcal{C}(\wt \Om)$ and $\mathcal{C}_{\mathbf{m}}(\wt \Om)$ defined by
\begin{align*}
	{\mathcal{C}}(\wt \Om):=\Large\{ (A,S_{h,K},u):   \, (A,S_{h,K}) \in {\mathcal{B}}(\wt \Om) \text{ and 
  $u \in  H^{1}_\mathrm{{loc}} ( \Int (A); \R^2)$}  \}
\end{align*}
and
\begin{equation*}
    \mathcal{C}_{\mathbf{m}}(\wt \Om): = \{(A,S_{h,K},u) \in{\mathcal{C}}(\wt \Om): \, (A,S_{h,K}) \in {\mathcal{B}}_{\mathbf{m}}(\wt \Om) \},
\end{equation*}
so that  $\mathcal{C}^1(\wt \Om)\subset \mathcal{C}(\wt \Om)$ and $\mathcal{C}^1_{\mathbf{m}}(\wt \Om)\subset \mathcal{C}_{\mathbf{m}}(\wt \Om)$. Therefore, since the elastic energy $\mathcal{W}$ and the surface energy  $\mathcal{S}$ considered in \cite{LlP} coincide with the energies $\mathcal{W}$ and $\mathcal{S}^1$ of this manuscript (by also observing that, following the notation of \cite{LlP}, $\varphi_0 = \varphi_{\mathrm{S}}, \varphi_1  = \varphi_{\mathrm{F}}, \varphi_{01}  = \varphi_{\mathrm{FS}}$, $\varphi^1_{01} = \varphi $ and $\varphi^2_{01} = \varphi'$), we have that
\begin{equation}
	\label{mleq:equivalenceenergiesS}
	\mathcal{S} \equiv \mathcal{S}^1 \quad \text{and} \quad \mathcal{F}\equiv \mathcal{F}^1,
\end{equation}
in $\mathcal{C}^1(\wt \Om)$ and $\mathcal{C}^1_{\mathbf{m}}(\wt \Om)$.
Finally, we also observe have that the topologies $\tau_{\B}$ and $\tau_{\C}$ defined in \cite{LlP} coincide with the topologies $\tau_{\B^1}$ and $\tau_{\C^1}$, respectively.

On the basis of these observations and by using the results for the two-phase setting of \cite{LlP}, 
we now prove that energy-equibounded sequences in $\mathcal{C}^1_{\bm{m}}$ are compact and that $\mathcal{F}^1$ is lower semicontinuous with respect to the topology $\tau_{\C^1}$.  

\begin{figure}[ht]
     \centering
     \includegraphics{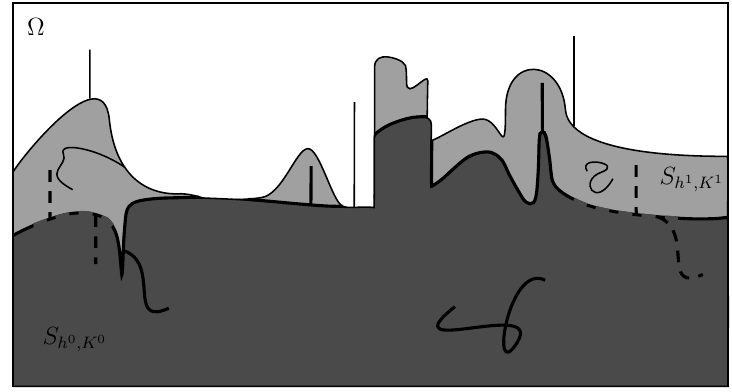}
     \caption{\small 
A single-layer film  (on the substrate 0th layer $S_{h^0,K^0}$) associated to an admissible configuration $(S_{h^1,K^1}, S_{h^0,K^0},u)\in\C^1_\textbf{m}$   (see Definition \ref{def:multilayers}) is represented by indicating each $j$th  layer with a gray color with decreasing value with respect to the increasing order of the index $j=0,1$ and each $j$th layer with a thinner line with respect to the increasing order of the index $j=0,1$. Furthermore, in the $0$th  layer we distinguish between its coherent and incoherent portions by using a dashed or a continuous line, respectively.
     }
     \label{mlfig:my_label2}
 \end{figure}

 \begin{proposition}
\label{mlcor:compactness}
Let $\{(S_{h^1_k,K^1_k}, S_{h^0_k,K^0_k} )\} \subset \mathcal{B}^1_\mathbf{m}$ be such that 
		\begin{equation}
  \label{eq:equibounded}
		   \sup_{k \in \N}{\left(\mathcal{S}^1(S_{h^1_k,K^1_k}, S_{h^0_k,K^0_k} ) {+  \mathcal{L}^2(S_{h^1_k,K^1_k})}\right)} < \infty .
		\end{equation}
	Then, there exist a not relabeled subsequence $\{(S_{h^1_k,K^1_k}, S_{h^0_k,K^0_k} )\} \subset \B^1_\mathbf{m}$ and $(S_{h^1,K^1}, S_{h^0,K^0}) \in \mathcal{B}^1_\mathbf{m}$ such that $(S_{h^1_k,K^1_k}, S_{h^0_k,K^0_k} )\xrightarrow{\tau_{\B^1}}(S_{h^1,K^1}, S_{h^0,K^0})$.   
 \end{proposition}
\begin{proof}

We begin by observing that in view of \cite[Theorem 3.47]{AFP}  from \eqref{eq:equibounded} it follows that there exists $\wt L>0$ such that for every $k \in \N$, $S_{h^1_k,K^1_k} \subset (-l,l)\times(-L, \wt L)=: \wt{\Om}$.
Since $(S_{h^1_k,K^1_k}, S_{h^0_k,K^0_k} ) \in \B^1_\textbf{m}$ for every $k \in \N$, by \eqref{mleq:equivalenceenergiesS}, where in \cite{LlP} we consider
$\varphi_0 = \varphi_{\mathrm{S}}, \varphi_1  = \varphi_{\mathrm{F}}, \varphi_{01}  = \varphi_{\mathrm{FS}}$, $\varphi^1_{01} = \varphi $ and $\varphi^2_{01} = \varphi'$, we have that  
   \begin{equation}         
   \label{mleq:corollarycompactness1} 
      \sup_{k \in \N}{\mathcal{S}(S_{h^1_k,K^1_k},S_{h^0_k,K^0_k})} = \sup_{k \in \N}{\mathcal{S}^1(S_{h^1_k,K^1_k}, S_{h^0_k,K^0_k} )} < \infty.
   \end{equation} 
By applying  \cite[Theorem 4.2]{LlP} with respect to the region $\wt \Om$, there exist a not relabeled subsequence $\{(S_{h^1_k,K^1_k},S_{h^0_k, K^0_k})\} \subset \B_\textbf{m}$ and $(S,S_{h^0,K^0}) \in  \B_\textbf{m}$ such that 
\begin{equation}
    \label{mleq:corollarycomp2}
    (S_{h^1_k,K^1_k},S_{h^0_k, K^0_k}) \xrightarrow{\tau_\B} (S,S_{h^0,K^0}).
\end{equation} 
By definition of $\tau_\B$-convergence and by the second statement of \cite[Lemma 3.8]{LlP} 
there exists $(h^1,K^1) \in \textrm{ AHK} (\Omega)$ such that 
\begin{equation}
    \label{mleq:corollarycomp2_bis}
    S= S_{h^1,K^1}.
 \end{equation}   
In view of the definition of $h^0$ and $h^1$ that comes from \cite[Lemma 3.8]{LlP}, we see that 
$$
    h^0(x_1):= \sup \{ \limsup_{k \to \infty} h^0_k (x_1^k) : x^k_1 \to x_1 \} \le  \sup \{ \limsup_{k \to \infty} h^1_k (x_1^k) : x^k_1 \to x_1 \} =: h^1(x_1)
$$
for every $x_1 \in [-l,l]$. Thus, $(S_{h^1,K^1}, S_{h^0,K^0}) \in \mathcal{B}^1_\mathbf{m}$. Finally, from \eqref{mleq:corollarycomp2} and \eqref{mleq:corollarycomp2_bis}, and by the fact that the $\tau_{\B^1}$-convergence is similar to the $\tau_{\B}$-convergence of \cite{LlP} we obtain that 
$$
(S_{h^1_k,K^1_k}, S_{h^0_k,K^0_k} )\xrightarrow{\tau_{\B^1}}(S_{h^1,K^1}, S_{h^0,K^0})
$$
which concludes the proof.
\end{proof}

We are in the position to prove that $\mathcal{C}^1_\mathbf{m}$ is compact with respect to the topology $\tau_{\mathcal{C}^1}$. 

	\cite[Theorem 4.3]{LlP}
	\begin{theorem}
	    [Compactness of $\mathcal{C}^1_\mathbf{m}$]
		\label{mlthm:compactnesslayers}
	Let $\{(S_{h^1_k,K^1_k}, S_{h^0_k,K^0_k}, u_k)\}_{k \in \N} \subset \mathcal{C}^1_\mathbf{m}$ be such that
		\begin{equation}
			\label{mleqthm:compactnesslayers}
			\sup_{k\in\mathbb{N}}{ \left( \mathcal{F}^1(S_{h^1_k,K^1_k}, S_{h^0_k,K^0_k}, u_k)  {+  \mathcal{L}^2(S_{h^1_k,K^1_k})} \right)} < \infty.
		\end{equation}
	Then, there exist an admissible configuration $(S_{h^1,K^1}, S_{h^0,K^0},u) \in \mathcal{C}^1_\mathbf{m}$ of finite $\mathcal{F}^1$ energy, a subsequence $\{( S_{h^1_\kn,K^1_\kn}, S_{h^0_\kn,K^0_\kn}, u_{k_n})\}_{n \in \N}$, a sequence $\{(S_{h^1_\kn, \wt K^1_n}, S_{h^0_\kn,\wt K^0_n} , u_\kn)\}_{n \in \N} \subset \mathcal{C}^1_\mathbf{m}$ and a sequence $\{b_n\}_{n\in \N}$ of piecewise rigid displacements associated to $S_{h^1_{k_n}, \widetilde{K}^1_n}$ such that 
 $$
 (S_{h^1_\kn, \wt K^1_n}, S_{h^0_\kn,\wt K^0_n} , u_\kn + b_n) \xrightarrow{\tau_{\mathcal{C}^1}} (S_{h^1,K^1}, S_{h^0,K^0},u)
 $$
 and 
		\begin{equation}\label{ml2liminfequalliminflayers}
			\liminf_{n \rightarrow \infty}{\mathcal{F}^1( S_{h^1_\kn,K^1_\kn}, S_{h^0_\kn,K^0_\kn}, u_{k_n}) } = \liminf_{n \rightarrow \infty}{\mathcal{F}^1(S_{h^1_\kn, \wt K^1_n}, S_{h^0_\kn,\wt K^0_n} , u_\kn + b_n) }.
		\end{equation}
	\end{theorem}
 
\begin{proof}
We begin by observing that in view of \cite[Theorem 3.47]{AFP}  from \eqref{eq:equibounded} it follows that  there exists $\wt L>0$ such that for every $k \in \N$, $S_{h^1_k,K^1_k} \subset (-l,l)\times(-L, \wt L)=: \wt{\Om}$.
In view of the observations at the beginning of the section by \eqref{mleq:equivalenceenergiesS} and \eqref{mleqthm:compactnesslayers} we have that
     $$
        \sup_{k\in\mathbb{N}}{ \mathcal{F}(S_{h^1_k,K^1_k} , S_{h^0_k,K^0_k}, u_k) } = \sup_{k\in\mathbb{N}}{ \mathcal{F}^1(S_{h^1_k,K^1_k}, S_{h^0_k,K^0_k}, u_k) } < \infty,
     $$
where $\mathcal{F}: \mathcal{C}(\wt \Om)\to [0,\infty]$ is the total energy considered in \cite{LlP}. Hence, by applying \cite[Theorem 4.3]{LlP} with respect to the region  $\wt \Om$, 
 we deduce that  there exist a triple $(S,S_{h^0,K^0},u) \in \mathcal{C}_\mathbf{m}(\wt{\Om})$ of finite $\mathcal{F}$-energy, a subsequence $\{(S_{h^1_{k_n}, K^1_\kn},S_{ h^0_{k_n},K^0_{k_n} }, u_{k_n} ) \}_{n \in \N}$, a sequence $\{ (\widetilde{S}_n, S_{h^0_{k_n},\widetilde{K}^0_n}, u_{k_n} ) \}_{n \in \N} \subset \mathcal{C}_\mathbf{m}$ and a sequence $\{b_n\}_{n\in \N}$ of piecewise rigid displacements associated to $\widetilde{S}_n$ such that 
 \begin{equation}
 \label{mlconvcor}
 (\widetilde{S}_n,S_{ h^0_{k_n}, \widetilde{K}^0_n}, u_{k_n} + b_n ) \xrightarrow{\tau_\mathcal{C}} (S,S_{h^0,K^0},u)
 \end{equation}
 and 
\begin{equation}\label{mlliminfequalliminfcor}
			\liminf_{n \rightarrow \infty}{\mathcal{F} (S_{h^1_{k_n}, K^1_\kn}, S_{h^0_{k_n},K^0_{k_n} },  u_{k_n} ) } = \liminf_{n \rightarrow \infty}{\mathcal{F} (\widetilde{S}_n, S_{h^0_{k_n}, \widetilde{K}^0_n}, u_{k_n} + b_n) }.
		\end{equation}
  In view of the proof of \cite[Theorem 4.3]{LlP} 
 we have that   
\begin{equation}\label{mldefAntilde}
        \wt{S}_n := S_{h^1_\kn,K^1_\kn} \setminus \left( \partial \wt{S}_n \setminus \partial S_{h^1_\kn,K^1_\kn}  \right)
  \end{equation}
  In analogy to the definition of $\wt{K}^0_n$ in the proof of \cite[Theorem 4.3]{LlP}, 
  we define
  \begin{equation}\label{mldefK1tilde}
  \wt{K}^1_n := K^1_n \cup  \left( \partial \wt{S}_n \setminus \partial S_{h^1_\kn,K^1_\kn}  \right)
  \end{equation}
  and we claim that $\wt{S}_n =S_{h^1_\kn, \wt{K}^1_n}$. Indeed,   we have that
  \begin{equation}
      \begin{split}
          \label{mleq:equivalencetildeK}
      S_{h^1_\kn, \wt{K}^1_n} &:= \partial S_{h^1_\kn} \cup \left( S_{h^1_\kn} \setminus  \wt{K}^1_n \right) =\partial S_{h^1_\kn} \cup \left( S_{h^1_\kn} \setminus  \left( K^1_n \cup  \left( \partial \wt{S}_n \setminus \partial S_{h^1_\kn,K^1_\kn}  \right) \right) \right) \\
      & = \partial S_{h^1_\kn} \cup  \left( S_{h^1_\kn} \cap  \left( (K^1_n)^c \cap  \left( \partial \wt{S}_n \setminus \partial S_{h^1_\kn,K^1_\kn}  \right)^c \right) \right) \\
      & = \left( \partial S_{h^1_\kn} \cup \left( S_{h^1_\kn} \setminus  {K}^1_n \right) \right) \cap \left( \partial S_{h^1_\kn} \cup  \left( \partial \wt{S}_n \setminus \partial S_{h^1_\kn,K^1_\kn}  \right)^c\right) \\
      &=: S_{h^1_\kn,K^1_\kn} \cap \left( \partial S_{h^1_\kn} \cup (\partial \wt{S}_n)^c\cup \partial S_{h^1_\kn, K^1_\kn}  \right) 
      = S_{h^1_\kn,K^1_\kn} \setminus  (\partial \wt{S}_n \setminus\partial S_{h^1_\kn, K^1_\kn}) \\
      & =: \wt{S}_n,
      \end{split}
  \end{equation}
  where we used \eqref{mleq:substrateheight} in the first and fifth equalities, \eqref{mldefK1tilde} in the second equality, De Morgan's laws in the third, fourth and sixth equalities, and \eqref{mldefAntilde} in the last equality, and hence  $( S_{h^1_\kn,\wt K^1_\kn}, S_{h^0_\kn, \wt K^0_\kn}  , u_{k_n}) \in \mathcal{C}^1_\textbf{m}$. 
  
  Furthermore, by \eqref{mleqthm:compactnesslayers} and the non-negativeness of $\mathcal{F}$ it follows from Proposition \ref{mlcor:compactness} that $S = S_{h^1,K^1}$ for a proper pair $(h^1,K^1) \in \text{AHK}(\Om)$ and, in particular, we have that
$(S_{ h^1,K^1}, S_{ h^0,K^0},u) \in \mathcal{C}^1_\textbf{m}$. Finally, in view of the definition of $\tau_{\mathcal{C}^1}$-convergence, by \eqref{mlconvcor} we obtain that 
 $$
 (S_{h^1_\kn,\wt K^1_\kn}, S_{h^0_\kn, \wt K^0_\kn}  , u_{k_n} + b_n) \xrightarrow{\tau_{\mathcal{C}^1}} (S_{h^1,K^1},S_{h^0,K^0} ,u), 
 $$
 and, by \eqref{mleq:equivalenceenergiesS}, 
 \eqref{mlliminfequalliminfcor} we obtain \eqref{ml2liminfequalliminflayers}, which  concludes the proof.
  \end{proof}

Now, by applying \cite[Theorem 5.14]{LlP} 
we prove that $\mathcal{F}^1$ is lower semicontinuous with respect to the $\tau_{\mathcal{C}^1}$-topology. 
	
	\begin{theorem}[Lower semicontinuity of $\mathcal{F}^1$]
		\label{mlthm:lowersemicontinuitylayer}
		Assume (H1)-(H3). Let $\{ (S_{h^1_k,K^1_k}, S_{h^0_k,K^0_k} , u_k) \}_{k \in \N} \subset \mathcal{C}^1_\mathbf{m}$ and $(S_{h^1,K^1}, S_{h^0,K^0} ,u) \in \mathcal{C}^1_\mathbf{m}$ be such that 
  $
  (S_{h^1_k,K^1_k}, S_{h^0_k,K^0_k} , u_k) \xrightarrow{\tau_{\mathcal{C}^1}} (S_{h^1,K^1}, S_{h^0,K^0} ,u) .
  $ 
  Then 
\begin{equation}\label{mleq:corlower0}
			\mathcal{F}^1  (S_{h^1,K^1}, S_{h^0,K^0} ,u) \le \liminf_{k \rightarrow \infty}{ \mathcal{F}^1 (S_{h^1_k,K^1_k}, S_{h^0_k,K^0_k} , u_k) }.
		\end{equation}
	\end{theorem}
 \begin{proof}
 Without loss of generality, we assume that the right side of \eqref{mleq:corlower0} is finite. In view of \cite[Theorem 3.47]{AFP} and from the fact that and since $
  (S_{h^1_k,K^1_k}, S_{h^0_k,K^0_k} , u_k) \xrightarrow{\tau_{\mathcal{C}^1}} (S_{h^1,K^1}, S_{h^0,K^0} ,u) 
  $it follows that there exists $\wt L>0$ such that for every $k \in \N$, $S_{h^1_k,K^1_k} \subset (-l,l)\times(-L, \wt L)=: \wt{\Om}$.
 Since the topology $\tau_{\B}$ considered in \cite{LlP} coincide with the topology $\tau_{\B^1}$, by \eqref{mleq:equivalenceenergiesS} and by applying \cite[Theorem 5.13]{LlP} applied in the regions $\Om = \wt \Om$, we have that 
\begin{equation}
    \label{mleq:corollarysemicontinuityS}
         \mathcal{S}^1 \left(S_{h^1,K^1},S_{h^0,K^0} \right)\le \liminf_{k \rightarrow \infty}{ \mathcal{S} ^1\left(S_{h^1_k,K^1_k}, S_{h^0_k,K^0_k} \right) }.
     \end{equation}
Now, we are going to prove that the elastic energy is lower semicontinuous. Indeed, let $D \subset \subset \textrm{Int}(S_{h^1,K^1})$, by properties of the signed distance convergence we have that 
$D \subset \subset \Int{(S_{h^1_k,K^1_k})}$ for $k$ large enough. By definition of $\tau_{\mathcal{C}^1}$ convergence we have that $u_k \rightarrow u$ a.e. in $D$. Furthermore, since $E \,u_k$ are bounded in the $L^2(D)$ norm, we have that $E\, u_k \rightharpoonup E\,u$ in $L^2 (D)$. By convexity of $W$ we obtain that
\begin{equation*}
    \int_{D} W(x, E\,u - {E_0^1}) \, dx \le \liminf_{k \rightarrow \infty} \int_{D} W(x, E\,u_k - { E_0^1}) \, dx \le \liminf_{k \rightarrow +\infty}{\mathcal{W}(S_{h^1_k,K^1_k}, S_{h^0_k,K^0_k}, u_k)}
\end{equation*}
By taking $D \nearrow \Int (S_{h^1,K^1})$ we conclude that 
\begin{equation}
    \label{eq:elasticitytwolayres}
    {\mathcal{W}(S_{h^1,K^1}, S_{h^0,K^0}, u)} \le \liminf_{k \rightarrow +\infty}{\mathcal{W}(S_{h^1_k,K^1_k}, S_{h^0_k,K^0_k}, u_k)}.
\end{equation}
By \eqref{mleq:corollarysemicontinuityS} and \eqref{eq:elasticitytwolayres} and thanks to the superadditivity of the \emph{liminf}, we get that
\begin{equation*}
\begin{split}
    \mathcal{F}^1  (S_{h^1,K^1}, S_{h^0,K^0} ,u) &:= {\mathcal{W}(S_{h^1,K^1}, S_{h^0,K^0}, u)} + \mathcal{S}^1 \left(S_{h^1,K^1},S_{h^0,K^0} \right) \\
    &\le \liminf_{k \rightarrow +\infty}{\mathcal{W}(S_{h^1_k,K^1_k}, S_{h^0_k,K^0_k}, u_k)} + \liminf_{k \rightarrow \infty}{ \mathcal{S} ^1\left(S_{h^1_k,K^1_k}, S_{h^0_k,K^0_k} \right) } \\
    & \le \liminf_{k \rightarrow +\infty} {\mathcal{W}(S_{h^1_k,K^1_k}, S_{h^0_k,K^0_k}, u_k) + \mathcal{S} ^1\left(S_{h^1_k,K^1_k}, S_{h^0_k,K^0_k} \right)} \\
    &=: \liminf_{k \rightarrow \infty}{ \mathcal{F}^1 (S_{h^1_k,K^1_k}, S_{h^0_k,K^0_k} , u_k) },
\end{split}
\end{equation*}
which concludes the proof.
 \end{proof}

 Finally, we state the main result of this section. The following result is the analogous result of Theorem \ref{mlthm:existence}, more precisely, we prove the existence of minimizers for a volume constrained problem and for an unconstrained problem, respectively, with respect to the admissible family of deformable film and substrate $\mathcal{C}^1_{\textbf{m}}$ for every $\textbf{m}:= (m_0, m_1)\in \N \times \N$.  

 \begin{theorem}[Existence of minimizers]
	\label{mlcor:maincorollary}
	Assume (H1)-(H3) and let $\mathbbm{v}_0, \mathbbm{v}_1 \in [\mathcal{L}^2({\Om}/2), \mathcal{L}^2({\Om})]$ such that $\mathbbm{v}_0 \le \mathbbm{v}_1$. Then for every $\mathbf{m}= (m_0,m_1) \in \N^2$, the volume constrained minimum problem 
	\begin{equation}
		\label{mleq:fconst}
		\inf_{(S_{h^1,K^1}, S_{h^0,K^0} ,u) \in \mathcal{C}^1_\mathbf{m},\, \mathcal{L}^2 ({S_{h^1, K^1}})= \mathbbm{v}_1,\, \mathcal{L}^2({S_{h^0,K^0}})= \mathbbm{v}_0}{\mathcal{F}^1(S_{h^1,K^1}, S_{h^0,K^0} ,u)}
	\end{equation}
	and the unconstrained minimum problem
	\begin{equation}
		\label{mleq:funcost}
		\inf_{(S_{h^1,K^1}, S_{h^0,K^0} ,u)\in \mathcal{C}^1_\textbf{m}}{\mathcal{F}^{1, \bm{\lambda}} (S_{h^1,K^1}, S_{h^0,K^0} ,u)}
	\end{equation}
	have solution, where $\mathcal{F}^{1, \bm{\lambda}}  : \mathcal{C}^1_\mathbf{m} \rightarrow \R$ is defined as
	\begin{equation*}
		\mathcal{F}^{1, \bm{\lambda}}(S_{h^1,K^1}, S_{h^0,K^0} ,u) := \mathcal{F}^1(S_{h^1,K^1}, S_{h^0,K^0} ,u)  + \sum_{i= 0}^1 \lambda_i \abs{ \mathcal{L}^2 (S_{h^i,K^i})- \mathbbm{v}_i }.
	\end{equation*}
	for any $\bm{\lambda} = (\lambda_0, \lambda_1) \in \R^2$ such that $\lambda_0, \lambda_1>0$.
\end{theorem}
	
\begin{proof}
		We follow the \emph{Direct Method of the Calculus of Variations}. Fix $\textbf{m}:=(m_1,m_0) \in \N^2$. Let $\{ (S_{h^1_k,K^1_k}, S_{h^0_k,K^0_k}, u_k)\} \subset \mathcal{C}^1_\mathbf{m}$ be a minimizing sequence of $\mathcal{F}^1$ such that $\mathcal{L}^2({S_{h^i_k,K^i_k}})= \mathbbm{v}_i$ for $i =0,1$, and
		$$
  \sup_{k\in\mathbb{N}} {\mathcal{F}^1 (S_{h^1_k,K^1_k}, S_{h^0_k,K^0_k}, u_k)}< \infty.
  $$
Since $\mathcal{L}^2({S_{h^1_k,K^1_k}})= \mathbbm{v}_1$ for every $k \in \N$,	by Theorem \ref{mlthm:compactnesslayers}, 
 there exist a subsequence $\{(S_{h^1_{k_l},K^1_{k_l}}, S_{h^0_{k_l},K^0_{k_l}} , u_{k_l})\}$, a sequence $\{(S_{h^1_{k_l}, \wt K^1_l}, S_{h^0_{k_l},\wt K^0_l} , v_l)\}_{n \in \N} \subset \mathcal{C}^1_\mathbf{m}$ and $(S_{h^1,K^1}, S_{h^0,K^0} ,u)  \in \mathcal{C}^1_\mathbf{m}$ such that 
 $$
 (S_{h^1_{k_l}, \wt K^1_l}, S_{h^0_{k_l},\wt K^0_l} , v_l) \xrightarrow{\tau_{\mathcal{C}^1}} (S_{h^1,K^1}, S_{h^0,K^0} ,u)$$
 as $l \to \infty$ and
		\begin{equation}
			\label{mleq:lexistence1}
			\liminf_{l\to \infty}{\mathcal{F}^1(S_{h^1_{k_l}, \wt K^1_l}, S_{h^0_{k_l},\wt K^0_l} , v_l)} = \liminf_{l \to \infty}{\mathcal{F}^1(S_{h^1_{k_l},K^1_{k_l}}, S_{h^0_{k_l},K^0_{k_l}} , u_{k_l})}.
		\end{equation} 
		According to Theorem \ref{mlthm:lowersemicontinuitylayer}, we have that
		\begin{equation}
			\label{mleq:lexistance2}
			\mathcal{F}^1(S_{h^1,K^1}, S_{h^0,K^0} ,u) \le \liminf_{l\to \infty}{\mathcal{F}^1(S_{h^1_{k_l}, \wt K^1_l}, S_{h^0_{k_l},\wt K^0_l} , v_l)}.
		\end{equation}
		We claim that $\{(S_{h^1,K^1}, S_{h^0,K^0} )\}$ and $(S_{h^1_{k_l}, \wt K^1_l}, S_{h^0_{k_l},\wt K^0_l} )$ satisfy the volume constraints of \eqref{mleq:fconst}. Indeed, fix $i =0,1$, by \cite[Theorem 4.3]{LlP}, 
  for any $l  \ge 1$, $\mathbbm{v}_i = \mathcal{L}^2({S_{h^i_{k_l}, K^i_{k_l}}}) = \mathcal{L}^2({S_{h^i_{k_l}, \wt K^i_{l}}})$. 
  Thanks to the fact that
$$
 (S_{h^1_{k_l}, \wt K^1_l}, S_{h^0_{k_l},\wt K^0_l} ) \xrightarrow{\tau_{\mathcal{B}^1}} (S_{h^1,K^1}, S_{h^0,K^0}),$$
applying \cite[Lemma 3.2]{KP} we infer that $S_{h^i_{k_l}, \wt K^i_{l}} \to S_{h^i,K^i}$ in $L^1(\R^2)$ as $l \to \infty$, and thus $\mathcal{L}^2({S_{h^i,K^i}}) = \mathbbm{v}_i$. From \eqref{mleq:lexistence1} and \eqref{mleq:lexistance2}, we deduce that
		\begin{equation*}
			\begin{split}
				&\inf_{(S_{h^1,K^1}, S_{h^0,K^0} ,u) \in \mathcal{C}^1_\mathbf{m},\, \mathcal{L}^2 ({S_{h^1, K^1}})= \mathbbm{v}_1,\, \mathcal{L}^2({S_{h^0,K^0}})= \mathbbm{v}_0}{\mathcal{F}^1(S_{h^1,K^1}, S_{h^0,K^0} ,u)}\\
    &\quad = \lim_{k \to \infty}\mathcal{F}^1 (S_{h^1_k,K^1_k}, S_{h^0_k,K^0_k}, u_k) \ge \liminf_{l\to \infty}{\mathcal{F}^1(S_{h^1_{k_l}, \wt K^1_l}, S_{h^0_{k_l},\wt K^0_l} , v_l)} \\
    & \quad \ge \mathcal{F}^1(S_{h^1,K^1}, S_{h^0,K^0} ,u).
			\end{split}
		\end{equation*}
		We conclude from the previous inequality that $(A,h,K,u)$ is a minimum of \eqref{mleq:fconst}. The same arguments are used to solve the unconstrained problem \eqref{mleq:funcost} by noticing that for a minimizing sequence $\{ (S_{h^1_k,K^1_k}, S_{h^0_k,K^0_k}, u_k)\} \subset \mathcal{C}^1_\mathbf{m}$ of $\mathcal{F}^{1,\bm{\lambda}}$ such that $$
  \sup_{k\in\mathbb{N}} {\mathcal{F}^{1,\bm{\lambda}} (S_{h^1_k,K^1_k}, S_{h^0_k,K^0_k}, u_k)}< \infty
  $$
  we have that 
  $$
    \mathcal{L}^2({S_{h^1_k,K^1_k}}) \le \abs{\mathcal{L}^2({S_{h^1_k,K^1_k}}) - \mathbbm{v}_1}  + \mathbbm{v}_1 \le \frac1{\lambda_1}  \mathcal{F}^{1,\bm{\lambda}} (S_{h^1_k,K^1_k}, S_{h^0_k,K^0_k}, u_k) + \mathbbm{v}_1,
  $$
  and thus $\sup_{k \in \N} \mathcal{L}^2({S_{h^1_k,K^1_k}})  < \infty$.
	\end{proof}

 \section{Multilayered films}
 \label{multisection}
In this section, we consider $\alpha>1$ and  we denote $\textbf{m}:= (m_0, \ldots,m_\alpha) \in \N^{\alpha+1}$. The main goal of this section is to prove Theorem \ref{mlthm:existence}. In order to do this, first we prove that $\C^\alpha_{{\bm{m}}}$ is compact and by induction, with respect to $\alpha \in \N$ we show that $\mathcal{F}^\alpha$ is lower semicontinuous with respect to the topology of $\tau_{\C^\alpha}$. Notice that in the previous section, we proved the basis of the induction for the lower semicontinuity property. 
We start by proving that $\B^\alpha_{{\bm{m}}}$ and $\C^\alpha_{\bm{m}}$ are compact.

 \begin{proposition}
\label{mlproo:compactness}
Let $\{(S_{h^\alpha_k,K^\alpha_k}, \ldots ,S_{h^0_k,K^0_k}) \} \subset \mathcal{B}^\alpha_{\mathbf{m}}$ such that 
		\begin{equation}
        \label{mllem:assumpition1}
		   \sup_{k \in \N}{\left(\mathcal{S}^\alpha_{ \sigma}(S_{h^\alpha_k,K^\alpha_k}, \ldots ,S_{h^0_k,K^0_k}) { + \mathcal{L}^2 (S_{h^\alpha_k,K^\alpha_k})}\right)} < \infty 
		\end{equation}
	Then, there exist a not relabeled subsequence $\{(S_{h^\alpha_k,K^\alpha_k}, \ldots ,S_{h^0_k,K^0_k})\} \subset \B^\alpha_{\mathbf{m}}$ and $\{(S_{h^\alpha, K^\alpha}, \ldots, S_{h^0,K^0} )\} \in \mathcal{B}^\alpha_{\mathbf{m}}$ such that $(S_{h^\alpha_k,K^\alpha_k}, \ldots ,S_{h^0_k,K^0_k})\xrightarrow{\tau_{\B^\alpha}} (S_{h^\alpha,K^\alpha}, \ldots ,S_{h^0,K^0})$.
 \end{proposition}
 \begin{proof}
     We proceed by induction on $\alpha \in \N$. If $\alpha = 1$ by Proposition \ref{mlcor:compactness} the assertion holds. We now prove that if the assertion holds for $\alpha = n$, then the assertion also holds for      
     $\alpha = n+1$. Thus, let us assume that the assertion holds for $\alpha = n$. 
     From the definition of $ \mathcal{S}^{n+1}_{\sigma}$ and the non-negativeness of  $\mathcal{S}^{(n+1,j)}$ for every $j=0, \ldots,n$ it follows that
     \begin{equation}
         \label{mleq:inductioncompactnessS1}
         \begin{split}
             \mathcal{S}^{n+1}_{ \sigma} (S_{h^{n+1}_k,K^{n+1}_k}, \ldots, S_{h^0_k,K^0_k} )& \geq \mathcal{S}^{n}_{ \sigma} (S_{h^{n}_k,K^{n}_k}, \ldots, S_{h^0_k,K^0_k} ) \\
             & \qquad +  \mathcal{S}^{(n+1,n)}(S_{h^{n+1}_k,K^{n+1}_k}, S_{h^n_k,K^n_k}) 
             \end{split}
     \end{equation}
     for every $k \in \N$.
     By \eqref{mllem:assumpition1}, \eqref{mleq:inductioncompactnessS1}, and the non-negativeness of $\mathcal{S}^{(n+1,n)}$  the induction hypothesis yields that, up to extracting a not relabeled subsequence,  there exists $(S_{h^{n},K^{n}}, \ldots, S_{h^0,K^0} ) \in \B^n_{\textbf{m}_\textbf{n}}$ such that $(S_{h^{n}_k,K^{n}_k}, \ldots, S_{h^0_k,K^0_k} )  \xrightarrow{\tau_{\B^n}} (S_{h^{n},K^{n}}, \ldots, S_{h^0,K^0} )$, where $\textbf{m}_\textbf{n} : = (m_0, \ldots, m_{n} )$. 
      Furthermore, by  \eqref{mllem:assumpition1} and   \eqref{mleq:inductioncompactnessS1} and  the non-negativeness of $\mathcal{S}^{n}_{ \sigma}$, in view of Proposition \ref{mlcor:compactness} applied to $\mathcal{S}^{(n+1,n)}$ it follows  that,  up to extracting a not relabeled subsequence, there exist  a region $S_{h^{n+1},K^{n+1}}$ with $(h^{n+1},K^{n+1})\in \text{AHK}(\Om)$
     (and $(S_{h^{n+1},K^{n+1}}, S_{h^n, K^n}) \in \B^1_{\textbf{m}_1}$ for ${\textbf{m}_1} : = (m_n, m_{n+1})\in \N^2$)  such that $(S_{h^{n+1}_k,K^{n+1}_k}, S_{h^n_k, K^n_k})\xrightarrow{\tau_{\B^1}} (S_{h^{n+1},K^{n+1}}, S_{h^n, K^n})$, where we used the uniqueness of the sign-distance convergence.

It remains to prove that $(S_{h^{n+1},K^{n+1}}, \ldots, S_{h^0,K^0} ) \in  \B^{n+1}_{\textbf{m}} $. Since $\partial S_{h^i,K^i}$ has at most $m_i$ connected components for $i = 0 , \ldots, n+1$, it remains only to check that $\partial S_{h^{n+1},K^{n+1}} \cap \Int(S_{h^n,K^n}) = \emptyset$. Let us assume by contradiction that $\partial S_{h^{n+1},K^{n+1}} \cap \Int(S_{h^n,K^n}) \neq \emptyset$. Then, there exists
\begin{equation}
     \label{mleq:cptsB2}
x \in  \partial S_{h^{n+1},K^{n+1}} \cap \Int(S_{h^n,K^n}).
\end{equation}
By properties of the signed distance convergence (see \cite[Remark 3.8]{LlP}) there exists $x_k \in \partial
S_{h^{n+1}_k,K^{n+1}_k}$ such that $x_k \to x$, and,  since $(S_{h^{n+1}_k,K^{n+1}_k}, S_{h^n_k, K^n_k})\xrightarrow{\tau_{\B^1}} (S_{h^{n+1},K^{n+1}}, S_{h^n, K^n})$
it follows that
\begin{equation}
     \label{mleq:cptsB3}
        \sdist(x, \partial S_{h^n_k,K^n_k}) \to \sdist(x, \partial S_{h^n,K^n}) \quad \text{as } k \to \infty.
     \end{equation}
By \eqref{mleq:cptsB2} there exists $\ep >0$ such that $\sdist(x , \partial S_{h^n,K^n})= -\ep$, we can find $k_0 := k_0(x)$ for which $\sdist(x, \partial S_{h^n_{k_0},K^n_{k_0}})$ is negative. Then, $x \in \Int(S_{h^n_{k_0},K^n_{k_0}})$ and so, there exists $\delta \le \ep/2$ such that
$$
	x_{k_0} \in B_{\delta}(x) \subset \Int(S_{h^n_{k_0},K^n_{k_0}}),
$$
which is an absurd since $\partial S_{h^{n+1}_{k_0},K^{n+1}_{k_0}} \cap \Int(S_{h^n_{k_0},K^n_{k_0}})= \emptyset$. 
Finally, we conclude the proof by observing that there exists 
$(S_{h^{n+1},K^{n+1}}, \ldots, S_{h^0,K^0} ) \in \B^{n+1}_{\textbf{m}}$ such that $(S_{h^{n+1}_k,K^{n+1}_k}, \ldots, S_{h^0_k,K^0_k} )  \xrightarrow{\tau_{\B^{n+1}}}(S_{h^{n+1},K^{n+1}}, \ldots, S_{h^0,K^0} )$. 
    \end{proof}

Now, we prove that $\mathcal{C}^\alpha_{{\bm m}}$ is compact with respect to the topology $\tau_{\mathcal{C}^\alpha}$.

\begin{proof}[Proof of Theorem \ref{mlcompactness:Clambda}]
		Denote $R:= \sup_{k \in \N} \left({ \mathcal{F}^\alpha_{\sigma} (S_{h^\alpha_k,K^\alpha_k}, \ldots ,S_{h^0_k,K^0_k}, u_k)} { + \mathcal{L}^2 (S_{h^\alpha_k,K^\alpha_k}) } \right) $. 
		Without loss of generality (by passing, if necessary, to a not relabeled subsequence), we assume that
		\begin{equation}\label{mllim_bound}
			\liminf_{k \rightarrow \infty}{\mathcal{F}^\alpha_{\sigma} (S_{h^\alpha_k,K^\alpha_k}, \ldots ,S_{h^0_k,K^0_k}, u_k)} = \lim_{k \rightarrow \infty}{\mathcal{F}^\alpha_{ \sigma} (S_{h^\alpha_k,K^\alpha_k}, \ldots ,S_{h^0_k,K^0_k}, u_k) }\leq R.
		\end{equation}
Since  $\mathcal{W}$ is a non-negative energy, by Proposition \ref{mlproo:compactness} there exist a  subsequence $\{(S_{h^\alpha_k,K^\alpha_k}, \ldots ,S_{h^0_k,K^0_k})\} \subset \B^\alpha_{\mathbf{m}}$ and $(S_{h^\alpha,K^\alpha}, \ldots ,S_{h^0,K^0}) \in \mathcal{B}^\alpha_{\mathbf{m}}$ such that 
$$
(S_{h^\alpha_k,K^\alpha_k}, \ldots ,S_{h^0_k,K^0_k}) \xrightarrow{\tau_{\B^\alpha}} (S_{h^\alpha,K^\alpha}, \ldots ,S_{h^0,K^0}).
$$ 
The rest of the proof is devoted to the construction of a sequence $\{(S_{h^\alpha_\kn,\wt K^\alpha_n}, \ldots, S_{h^0_\kn,\wt K^0_n} , u_{k_n})\} \subset \B^\alpha_{\mathbf{m}}$ to which we can apply \cite[Corollary 3.8]{KP} (with $P=\Int{(S_{h^\alpha,K^\alpha})}$ and $P_n=\Int{(S_{h^\alpha_\kn,\wt K^\alpha_n})}$, respectively)  in order to obtain $u\in H^1_{{\rm loc}}(\Int(S_{h^\alpha,K^\alpha});\R^2)$ such that $(S_{h^\alpha,K^\alpha}, \ldots ,S_{h^0,K^0},u) \in \mathcal{C}^\alpha_{\mathbf{m}}$ has finite energy, and a sequence $\{b_n\}_{n\in \N}$ of piecewise rigid displacements such that 
$$
(S_{h^\alpha_\kn,\wt K^\alpha_n}, \ldots, S_{h^0_\kn,\wt K^0_n} , u_{k_n} + b_n) \xrightarrow{\tau_{\mathcal{C}^\alpha}} (S_{h^\alpha,K^\alpha}, \ldots ,S_{h^0,K^0},u).
$$
Furthermore, we observe that also Equation \eqref{mlliminfequalliminflayers} 
will be a consequence of such construction and hence, the assertion of the theorem will directly follow.

By \cite[Proposition 3.6]{KP} applied to $S_{h^\alpha_\kn,K^\alpha_\kn}$ and $S_{h^\alpha,K^\alpha}$ there exist a  not relabeled subsequence $\{S_{h^\alpha_\kn,K^\alpha_\kn}\}$ and a sequence $\{\widetilde{A}_n\}$ with $\Hs^1$-rectifiable boundary $\partial \widetilde{A}_n$ of at most $m_\alpha$-connected components such that  
\begin{equation}
    \label{mleq:boundpartialtildeA}
    \sup_{n\in \N} \Hs^1 (\partial \widetilde{A}_n) < \infty,
\end{equation} that satisfy the following  properties:
\begin{itemize}
	\item[(a1)] $\partial S_{h^\alpha_\kn,K^\alpha_\kn} \subset \partial \widetilde{A}_n$ and $\displaystyle \lim_{n \to \infty} \Hs^1 (\partial \widetilde{A}_n \setminus \partial S_{h^\alpha_\kn,K^\alpha_\kn}) = 0$,
	\item[(a2)]  $\sdist(\cdot, \partial \widetilde{A}_n) \to \sdist(\cdot, \partial S_{h^\alpha,K^\alpha})$ locally uniformly in $\R^2$ as $n \to \infty$,
	\item[(a3)] If $\{E_i\}_{i \in I}$ is the family of all connected components of ${\rm{Int}}(S_{h^\alpha,K^\alpha}),$ we can find the connected components of ${\rm{Int} }(\widetilde{A}_n)$, which we enumerate as $\{E_i^n\}_{i \in I}$, such that for any $i$ and $G \subset \subset E_i$ one has $G \subset \subset E_i^n$ for all $n$ large (depending only on $i$ and $G$),
	\item[(a4)] $\mathcal{L}^2 (\widetilde{A}_n) = \mathcal{L}^2(S_{h^\alpha_\kn,K^\alpha_\kn})$.
\end{itemize}
Furthermore, from the construction of $\widetilde{A}_n$ (namely from the fact that $\widetilde{A}_n$ is constructed by adding extra ``internal'' topological boundary to the selected subsequence $S_{h^\alpha_\kn,K^\alpha_\kn}$,  see \cite[Propositions 3.4 and 3.6]{KP}) it follows  that 
\begin{equation}\label{ml1fromproof3.6}
\widetilde{A}_n = S_{h^\alpha_\kn,K^\alpha_\kn} \setminus (\partial \widetilde{A}_n \setminus \partial S_{h^\alpha_\kn,K^\alpha_\kn})
\end{equation}
with $\partial \widetilde{A}_n \setminus \partial S_{h^\alpha_\kn,K^\alpha_\kn}$ given by a finite union of closed 
$\Hs^1$-rectifiable sets  
connected to $\partial S_{h^\alpha_\kn,K^\alpha_\kn}$. More precisely, there exist a finite index set $J$ and a family $\{\Gamma_j\}_{j\in J}$ of closed $\Hs^1$-rectifiable sets of $\Om$ connected to $\partial S_\kn$ such that
$$
   \partial \widetilde{A}_n \setminus \partial S_{h^\alpha_\kn,K^\alpha_\kn} = \bigcup_{j \in J} \Gamma_j. 
$$ 
We define 
$$ \widetilde{K}^i_n := K^i_{k_n} \cup ((\partial \widetilde{A}_n \setminus \partial S_{h^\alpha_\kn,K^\alpha_\kn}) \cap S_{h^i_\kn}) \subset S_{h^i_\kn},$$
for every $i = 0 , \ldots, \alpha$, and we observe that $\widetilde{K}^i_n$ is closed and $\Hs^1$-rectifiable in view of the fact that  $\partial \widetilde{A}_n \setminus \partial S_{h^\alpha_\kn,K^\alpha_\kn}$ is a closed set in $\Omega$ and is $\Hs^1$-rectifiable, since  $\partial \widetilde{A}_n$ is $\Hs^1$-rectifiable. Therefore, $(h^i_\kn, \widetilde{K}^i_n) \in \text{AHK}(\Om)$ for every $i =0,\ldots,\alpha$. Furthermore, we have that 
$$
S_{h^i_{k_n}, \widetilde{K}^i_n}\subset S_{h^i_{k_n}} \subset S_{h^j_{k_n}} =  \overline{S_{h^j_{k_n},\widetilde{K}^j_n }},
$$
for every $0\le i\le j \le \alpha$. 
We claim that $\partial S_{h^i_{k_n}, \widetilde{K}^i_n}$ has at most $m_i$-connected components for $i = 0 , \ldots, \lambda$. Indeed, let $i \in \{ 0 , \ldots, \alpha\}$, if for every $j \in J$, $S_{h^i_\kn,K^i_\kn} \cap \Gamma_j$ is empty there is nothing to prove, so we assume that there exists $j \in J$ such that $S_{h^i_\kn,K^i_\kn} \cap \Gamma_j \neq \emptyset$. On one hand if $\Gamma_j \subset S_{h^i_\kn,K^i_\kn}$, thanks to the facts that $\Gamma_j$ is connected to $\partial S_{h^\alpha_\kn,K^\alpha_\kn}$ and $S_{h^i_\kn,K^i_\kn} \subset S_{h^\lambda_\kn}$, we 
deduce that $\Gamma_j$ needs to be connected to $\partial S_{h^\alpha_\kn,K^\alpha_\kn}$. On the other hand, if $\Gamma_j \cap ( S_{h^\alpha_\kn,K^\alpha_\kn} \setminus S_{h^i_\kn}) \neq \emptyset$, then we can find $x_1 \in \Gamma_j \cap S_{h^i_\kn,K^i_\kn} $ and $x_2  \in \Gamma_j \cap (S_{h^\alpha_\kn,K^\alpha_\kn} \setminus S_{h_\kn})$. Since $\Gamma_j$ is closed and 
connected, by \cite[Lemma 3.12]{F} there exists a parametrization $r:[0,1]\to\mathbb{R}^2$ whose support $\gamma \subset \Gamma_j$ joins the point $x_1$ with $x_2$. Thus, $\gamma$ crosses $\partial S_{h^i_\kn,K^i_\kn}$ and we conclude that $\Gamma_j$ is connected to $\partial S_{h^i_\kn,K^i_\kn}$. Finally, by repeating the same arguments of \eqref{mleq:equivalencetildeK}, we obtain that
$$
    \wt A_n =  S_{h^\alpha_\kn,\wt K^\alpha_\kn}
$$ 
and thus, $(S_{h^\alpha_\kn,\wt K^\alpha_n}, \ldots, S_{h^0_\kn,\wt K^0_n} ) \in \B^\alpha_{\textbf{m}}$.

We claim that $(S_{h^\alpha_\kn,\wt K^\alpha_n}, \ldots, S_{h^0_\kn,\wt K^0_n} ) \xrightarrow{\tau_{\B^\alpha}} (S_{h^\alpha,K^\alpha}, \ldots ,S_{h^0,K^0})$ as $n \to \infty$. In view of \eqref{mleq:boundpartialtildeA}, (a2), by \eqref{mleq:uniongraphs} and the previous construction of $\widetilde{K}^i_n$, 
$$
\sup_{n \in \N} \Hs^1(\partial S_{h^i_\kn, \widetilde{K}^i_n}) < \infty,
$$
for every $i = 0, \ldots, \lambda$. 
It remains to prove that 
\begin{equation}\label{mlclaim_convergence}
\sdist(\cdot, \partial S_{h_\kn^i, \widetilde{K}_n^i}) \to \sdist(\cdot, \partial S_{h^i,K^i})
\end{equation} 
locally uniformly in $\R^2$ as $n \to \infty$ for every $i = 0, \ldots, \alpha$. Let us fix $i = 0, \ldots, \alpha$,  by properties of the signed distance convergence, 
it suffices to prove that $S_{h^i_\kn, \widetilde{K}^i_n} \cngK S_{h^i}$ and that $\Om \setminus S_{h^i_\kn, \widetilde{K}^i_n} \cngK \Om \setminus \Int(S_{h^i,K^i})$. 
On one hand, by the $\tau_{\B^\alpha}$-convergence of $\{(S_{h^\alpha_\kn,\wt K^\alpha_n}, \ldots, S_{h^0_\kn,\wt K^0_n} ) \}$, the fact that $\overline{S_{h^i_\kn, \widetilde{K}^i_n}} = S_{h^i_\kn}$, and the properties of Kuratowski convergence, it follows that $S_{h^i_\kn, \widetilde{K}^i_n} \cngK S_{h^i}$. On the other hand, let $x \in \Om \setminus \Int(S_{h^i,K^i})$, since 
$$
    \Int(S_{h^i_\kn, \widetilde{K}^i_n} ) = \Int(S_{h^i_\kn}) \setminus\widetilde{K}^i_n \subset \Int(S_{h^i_\kn}) \setminus K^i_\kn = \Int(S_{h^i_\kn, K_\kn} )
$$
and by the fact that $\Om\setminus \Int(S_{h^i_\kn, K^i_\kn}) \cngK \Om \setminus \Int(S_{h^i,K^i})$, there exists 
$$x_n \in \Om \setminus \Int(S_{h^i_\kn,K^i_\kn}) \subset \Om \setminus \Int(S_{h^i_\kn,\widetilde{K}^i_n}) $$
such that $x_n \to x$. Now, we consider a sequence  $x_n \in \Om \setminus \Int(S_{h^i_\kn,\widetilde{K}^i_n})$ converging to a point $x\in\Omega$. 
We proceed by contradiction, namely we assume that $x \in \Int(S_{h^i,K^i})$. Therefore, there exists $\epsilon >0$ such that $\sdist(x, \partial S_{h^i,K^i}) = - \epsilon$, which implies that $\sdist(x, \partial S_{h^i_\kn,K^i_\kn}) \to -\epsilon$ as $n \to \infty$. Thus, there exists $n_\epsilon \in \N$, such that $x_n \in B_{\epsilon/2}(x) \subset \Int(S_{h^i_\kn,K^i_\kn})$, for every $n \ge n_\epsilon$. However, notice that
\begin{equation}
    \label{mleq:compactnessinteriorS}
    \begin{split}
        x_n \in \Om \setminus \Int\left(S_{h^i_\kn, \widetilde{K}^i_n}\right) &= \Om \setminus \left( \Int({S_{h^i_\kn}}) \setminus \widetilde{K}^i_n \right) \\
        &= \left( \Om \setminus \Int\left(S_{h^i_\kn,K^i_\kn}\right)\right) \cup  \left( \left( \partial S_{h^\alpha_\kn,\wt K^\alpha_\kn} \setminus \partial S_{h^\alpha_\kn, K^\alpha_\kn} \right) \cap S_{h^i_\kn} \right),
    \end{split}
\end{equation}
where in the last equality we used the definition of $\widetilde{K}^i_n := K^i_\kn \cup ( ( \partial S_{h^\alpha_\kn,\wt K^\alpha_\kn} \setminus \partial S_{h^\alpha_\kn, K^\alpha_\kn} ) \cap S_{h^i_\kn} )$ and the fact that $\Int(S_{h^i_\kn,K^i_\kn}) = \Int(S_{h^i_\kn}) \setminus K^i_\kn$. Therefore, by \eqref{mleq:compactnessinteriorS} we deduce that $x_n \in  \partial S_{h^\alpha_\kn,\wt K^\alpha_\kn} \setminus \partial S_{h^\alpha_\kn, K^\alpha_\kn}$ for every $n \ge n_\epsilon$ and hence, $x \in\partial S_{h^\alpha, K^\alpha}$ by (a2) and by \cite[Remark 3.7]{LlP}. 
We reached an absurd as it follows that $x\in \Int (S_{h^i,K^i})\cap\partial S_{h^\alpha, K^\alpha}=\emptyset$. This concludes the proof of \eqref{mlclaim_convergence} and hence, of the claim.

By \eqref{mleq:H1} and by conditions (a1), (a4) and \eqref{ml1fromproof3.6}, we observe that
\begin{equation}
	\label{mleq:compactnessIntS1}
	\lim_{n \to \infty} {\abs{\mathcal{S}^\alpha_{
 \sigma}(S_{h^\alpha_\kn,K^\alpha_\kn}, \ldots ,S_{h^0_\kn,K^0_\kn}) - \mathcal{S}^\alpha_{
 \sigma}  (S_{h^\alpha_\kn,\wt K^\alpha_n}, \ldots, S_{h^0_\kn,\wt K^0_n} ) }} = 0,
\end{equation}
and
\begin{equation}
	\label{mleq:compactnessIntW}
	\mathcal{W}(S_{h^\alpha_\kn,K^\alpha_\kn}, \ldots ,S_{h^0_\kn,K^0_\kn}, u_\kn) = \mathcal{W}(S_{h^\alpha_\kn,\wt K^\alpha_n}, \ldots, S_{h^0_\kn,\wt K^0_n},u_\kn). 
\end{equation}
By \eqref{mleq:H3}, \eqref{mllim_bound},  \eqref{ml1fromproof3.6}, \eqref{mleq:compactnessIntW}, (a3) and thanks to the fact that $\mathcal{S}^\alpha_{\sigma}$ is non-negative, we obtain that
		\begin{equation*}		\int_{E_i^n}{\abs{e(u_\kn)}^2 dx} \le \int_{S_{h^\alpha_\kn, \wt K^\alpha_\kn}}{\abs{e(u_\kn)}^2 dx} \le C\frac{R}{2c_3},
		\end{equation*}
for every $i\in I$, for $n$ large enough and for a constant $C:= C(u_0^1, \ldots, u_0^\alpha) >0$. Therefore, by a diagonal argument and by \cite[Corollary 3.8]{KP} (applied to, with the notation of \cite{KP}, $ P = E_i$ and $P_n = E_i^n$) up to extracting not relabelled subsequences both for  $\{u_{\kn}\} \subset H^1_{{\rm loc}}(\Om; \R^2)$ and $\{E_i^{n}\}_n $ there exist $w_i \in H^1_{\rm{loc}} (E_i, \R^2)$,  and a sequence of rigid displacements $\{b_n^i\}$ such that $(u_\kn + b_n^i) \mathbbm{1}_{E^{n}_i} \rightarrow w_i$ a.e. in $E_i$. Let $\{D^n_i\}_{i \in \widetilde{I}}$ for an index set $\widetilde{I}$ be the family of open and connected components of $S_{h^\alpha_\kn, \wt K^\alpha_\kn} \setminus \bigcup_{i\in I}{E_i^n}$ such that by (a3) $\Int(D^n_i)$ converges to the empty set for every $i \in \widetilde{I}$. In $D_i^n$ we consider the null rigid displacement, and we define
		\begin{equation*}
           {b}_n := \sum_{i \in I}{b_n^i} \mathbbm{1}_{E_i^n}  \quad \text{and} \quad u := \sum_{i \in I}{w_i \mathbbm{1}_{E_i}}.
		\end{equation*}
	We have that $u \in H^1_{\rm{loc}}({\rm{Int}} (S_{h^\alpha, K^\alpha}); \R^2 )$, ${b}_n$ is a rigid displacement associated to $S_{h^\alpha_\kn, \wt K^\alpha_\kn}$,  $u_\kn + {b}_n \rightarrow u$ a.e. in ${\rm{Int}}(S_{h^\alpha, K^\alpha})$ and hence,  $ (S_{h^\alpha,K^\alpha}, \ldots ,S_{h^0,K^0}, u) \in \mathcal{C}^\alpha_{\mathbf{m}}$ and $(S_{h^\alpha_\kn,\wt K^\alpha_n}, \ldots, S_{h^0_\kn,\wt K^0_n} ,  u_\kn + {b}_n) \xrightarrow{\tau_{\mathcal{C}}} (S_{h^\alpha,K^\alpha}, \ldots ,S_{h^0,K^0}, u)$.
	Furthermore, as $E(u_\kn + b_n) = Eu_\kn$, from 
	\eqref{mleq:compactnessIntS1} and \eqref{mleq:compactnessIntW} it follows that
\begin{equation}
	\label{mleq:compactnessIntS3}
	\lim_{n \to \infty} {\abs{\mathcal{F}^\alpha_{ \sigma} (S_{h^\alpha_\kn,K^\alpha_\kn}, \ldots, S_{h^0_\kn,K^0_\kn}  ,u_\kn) - \mathcal{F}^\alpha_{ \sigma} (S_{h^\alpha_\kn,\wt K^\alpha_n}, \ldots, S_{h^0_\kn,\wt K^0_n} , u_\kn + b_n) }} = 0, 
\end{equation}
	 which implies \eqref{mlliminfequalliminflayers} and completes the proof.
	\end{proof}
  
In the following proof, we show by induction that $\mathcal{F}^\alpha_{\sigma}$ is lower semicontinuous.
	
\begin{proof}[Proof of Theorem \ref{mlthm:lowersemicontinuitymultilayer}]
Since 
$$
 \mathcal{F}^\alpha_{ \sigma} (S_{h^\alpha_k,K^\alpha_k}, \ldots, S_{h^0_k,K^0_k} , u_k): = \mathcal{S}^\alpha_{ \sigma} (S_{h^\alpha_k,K^\alpha_k}, \ldots, S_{h^0_k,K^0_k}) + \mathcal{W}( S_{h^\alpha_k,K^\alpha_k}, \ldots, S_{h^0_k,K^0_k}, u_k),
$$
and by non-negativeness of $\mathcal{S}^\alpha_{\sigma}$ and $\mathcal{W}$ we prove first that $\mathcal{S}^\alpha_{ \sigma}$ is lower semicontinuous with respect to the convergence in $\tau_{\B^\alpha}$, and then we prove that $\mathcal{W}$ is lower semicontinuous with respect to the convergence in $\tau_{\mathcal{C}^\alpha}$.

First, we assume that $\sigma = t$. To prove that $\mathcal{S}^\alpha_{t}$ is lower semicontinuous we proceed by induction on $\alpha \in \N$. Notice that if $\alpha = 1$, by Theorem \ref{mlthm:lowersemicontinuitylayer} the assertion holds. Assume now that for $\alpha = n$ the assertion of the theorem holds. We are going to prove that the assertion of the theorem holds if $\alpha = n+1$.  By definition of the energy $\mathcal{S}^{n+1}$ we see that
\begin{equation}
         \label{mleq:inductionlwerS1}
         \begin{split}
             \mathcal{S}^{n+1}_{t} (S_{h^{n+1}_k,K^{n+1}_k}, \ldots, S_{h^0_k,K^0_k} )&=  \mathcal{S}^{n}_{t} (S_{h^{n}_k,K^{n}_k}, \ldots,S_{h^0_k,K^0_k} ) \\ 
             & \quad + \sum_{j=0}^{n} \mathcal{S}^{(n+1,j)}(S_{h^{n+1}_k,K^{n+1}_k}, S_{h^j_k,K^j_k} )
         \end{split}
     \end{equation}
for every $k \in \N$. By the induction hypothesis and the fact that 
$$
  (S_{h^{n}_k,K^{n}_k}, \ldots,S_{h^0_k,K^0_k}) \xrightarrow{\tau_{\mathcal{B}^n}} (S_{h^{n},K^{n}}, \ldots,S_{h^0,K^0})
$$
we have that
\begin{equation}
    \label{mleq:inductionlowerS2}
    \mathcal{S}^{n}_{t} (S_{h^{n},K^{n}}, \ldots,S_{h^0,K^0}) \le \liminf_{k \to \infty} \mathcal{S}^{n}_{ t}(S_{h^{n}_k,K^{n}_k}, \ldots,S_{h^0_k,K^0_k}).
\end{equation}
Furthermore, by the definition of $\tau_{\B^{n+1}}$-convergence it follows that
\begin{equation}
    \label{mleq:inductionlowerS3}
    (S_{h^{n+1}_k,K^{n+1}_k}, S_{h^j_k,K^j_k} ) \xrightarrow{\tau_{\mathcal{B}^1}}  (S_{h^{n+1},K^{n+1}}, S_{h^j_k,K^j_k} )
\end{equation}
for every $j = 0 , \ldots, n$. Thus, by \eqref{mleq:inductionlowerS3} and by the basis of the induction (see the proof Theorem \ref{mlthm:lowersemicontinuitylayer}) we deduce that
\begin{equation}
    \label{mleq:inductionlowerS4}
    \mathcal{S}^{(n+1,j)}(S_{h^{n+1},K^{n+1}}, S_{h^j_k,K^j_k} ) \le \liminf_{k \to \infty} \mathcal{S}^{(n+1,j)}(S_{h^{n+1}_k,K^{n+1}_k}, S_{h^j_k,K^j_k} )
\end{equation}
for every $j = 0 , \ldots, n$. 
It follows from the superadditivity of the \emph{liminf} that
\begin{equation}
    \label{mleq:inductionlowerS5} 
    \begin{split}
        \sum_{j=0}^{n}  \mathcal{S}^{(n+1,j)}(S_{h^{n+1},K^{n+1}}, S_{h^j_k,K^j_k} ) &\le \sum_{j=0}^{n} \liminf_{k \to \infty}  \mathcal{S}^{(n+1,j)}(S_{h^{n+1}_k,K^{n+1}_k}, S_{h^j_k,K^j_k} )\\
        & \le \liminf_{k \to \infty}  \sum_{j=0}^{n}  \mathcal{S}^{(n+1,j)}(S_{h^{n+1}_k,K^{n+1}_k}, S_{h^j_k,K^j_k} ),
    \end{split}
\end{equation}
where in the first inequality we used \eqref{mleq:inductionlowerS4}. Therefore, we have that
\begin{equation}
    \label{mleq:inductionlowerS6}
    \begin{split}
              \mathcal{S}^{n+1}_{t} (S_{h^{n+1},K^{n+1}}, \ldots, S_{h^0,K^0} )&= \mathcal{S}^{n}_{ t} (S_{h^{n},K^{n}}, \ldots,S_{h^0,K^0} ) \\ 
             & \quad + \sum_{j=0}^{n} \mathcal{S}^{(n+1,j)}(S_{h^{n+1},K^{n+1}}, S_{h^j,K^j} )
 \\
             & \le  \liminf_{k \to \infty} \mathcal{S}^{n}_{ t}(S_{h^{n}_k,K^{n}_k}, \ldots,S_{h^0_k,K^0_k}) \\
             & \quad + \liminf_{k \to \infty}  \sum_{j=0}^{n}  \mathcal{S}^{(n+1,j)}(S_{h^{n+1}_k,K^{n+1}_k}, S_{h^j_k,K^j_k} ) \\
             & \le  \liminf_{k \to \infty} \left(   \mathcal{S}^{n}_{t}(S_{h^{n}_k,K^{n}_k}, \ldots,S_{h^0_k,K^0_k}) \right. \\
             & \quad \left.+  \sum_{j=0}^{n}  \mathcal{S}^{(n+1,j)}(S_{h^{n+1}_k,K^{n+1}_k}, S_{h^j_k,K^j_k} ) \right) \\
             &=  \liminf_{k \to \infty} \mathcal{S}^{n+1}_{t} (S_{h^{n+1}_k,K^{n+1}_k}, \ldots, S_{h^0_k,K^0_k} ),
    \end{split}
\end{equation}
where in the first and the second equality we used \eqref{mleq:inductionlwerS1}, in the first inequality we used \eqref{mleq:inductionlowerS2} and \eqref{mleq:inductionlowerS5}, and in the second inequality we used the superadditivity of the \emph{liminf}, and thus, $\mathcal{S}^\alpha$ is lower semicontinuous with respect to the topology $\tau_{\mathcal{B}^\alpha}$. 

For the case with $\sigma = s$, we proceed also by induction and, on the basis of  the induction hypothesis with $\sigma = s$,  in order to obtain the lower semicontinuous property for $\mathcal{S}^{n+1}_{s}$, we proceed in the same way, with the only difference that the subscript $t$ is replaced by $s$, instead of \eqref{mleq:inductionlwerS1} we use
\begin{equation*}
         \begin{split}
             \mathcal{S}^{n+1}_{s} (S_{h^{n+1}_k,K^{n+1}_k}, \ldots, S_{h^0_k,K^0_k} )&=  \mathcal{S}^{n}_{s} (S_{h^{n}_k,K^{n}_k}, \ldots,S_{h^0_k,K^0_k} ) \\ 
             & \quad +  \mathcal{S}^{(n,n+1)}(S_{h^{n+1}_k,K^{n+1}_k}, S_{h^n_k,K^n_k} ),
         \end{split}
     \end{equation*}
and that, because of the definition of $\mathcal{S}^{n+1}_{s}$, it is enough to use \eqref{mleq:inductionlowerS4}  only with $j=n$.

With respect to the elastic energy it is enough to repeat the same arguments of the proof of lower semicontinuity of $\mathcal{W}$ in the proof of Theorem \ref{mlthm:lowersemicontinuitylayer}, from which we obtain that
\begin{equation}
    \label{mleq:lowersemiW}
    \mathcal{W}(S_{h^\alpha_k,K^\alpha_k}, \ldots, S_{h^0_k,K^0_k}, u) \le \liminf_{k \to \infty} \mathcal{W}(S_{h^\alpha_k,K^\alpha_k}, \ldots, S_{h^0_k,K^0_k}, u_k).  
\end{equation}

Finally, we conclude the proof by observing that by the superadditivity of the \emph{liminf} it follows that
\begin{equation*}
    \begin{split}
        \mathcal{F}^ \alpha_{\sigma}(S_{h^\alpha,K^\alpha}, \ldots ,S_{h^0,K^0}, u)  & := \mathcal{S}^\alpha_{\sigma}(S_{h^\alpha,K^\alpha}, \ldots ,S_{h^0,K^0}) +  \mathcal{W}(S_{h^\alpha_k,K^\alpha_k}, \ldots, S_{h^0_k,K^0_k}, u) \\
        & \le \liminf_{k \to \infty} \mathcal{S}^\alpha_{\sigma}(S_{h^\alpha_k,K^\alpha_k}, \ldots ,S_{h^0_k,K^0_k}) \\
        & \quad \quad + \liminf_{k \to \infty} \mathcal{W}(S_{h^\alpha_k,K^\alpha_k}, \ldots, S_{h^0_k,K^0_k}, u_k)  \\
        & \le \liminf_{k \to \infty} \left(  \mathcal{S}^\alpha_{\sigma}(S_{h^\alpha_k,K^\alpha_k}, \ldots ,S_{h^0_k,K^0_k}) + \mathcal{W}(S_{h^\alpha_k,K^\alpha_k}, \ldots, S_{h^0_k,K^0_k}, u_k)  \right) \\
        &=: \liminf_{k \to \infty} \mathcal{F}^\alpha_{\sigma} (S_{h^\alpha_k,K^\alpha_k}, \ldots ,S_{h^0_k,K^0_k}, u_k),
    \end{split}
\end{equation*}
where in the first inequality we used the lower semicontinuity of $\mathcal{S}^\alpha$ and $\mathcal{W}$.
\end{proof}

Finally, we are now in a position to prove the main result of this paper.

  \begin{proof}[Proof of Theorem \ref{mlthm:existence}]
		We follow the \emph{Direct Method of the Calculus of Variations}. Fix $\mathbf{m}= (m_0, \ldots,m_\alpha) \in \N^{\alpha+1}$ and let $\{(S_{h^\alpha_k,K^\alpha_k}, \ldots ,S_{h^0_k,K^0_k}, u_k)\} \subset \mathcal{C}^\alpha_{\mathbf{m}}$ be a minimizing sequence of $\mathcal{F}^\alpha_{\sigma}$ such that $\mathcal{L}^2({S_{h^i_k,K^i_k}})= \mathbbm{v}_i$ for $i = 0, \ldots, \alpha$, and
		$$\sup_{k\in\mathbb{N}} {\mathcal{F}^\alpha_{\sigma}(S_{h^\alpha_k,K^\alpha_k}, \ldots ,S_{h^0_k,K^0_k}, u_k)}< \infty.$$
		Since $\mathcal{L}^2({S_{h^\alpha_k,K^\alpha_k}})= \mathbbm{v}_\alpha$, by Theorem \ref{mlcompactness:Clambda} there exist a subsequence $\{(S_{h^\alpha_\kn,K^\alpha_\kn}, \ldots ,S_{h^0_\kn,K^0_\kn}, u_\kn)\}$, a sequence  $\{(S_{h^\alpha_\kn,\wt K^\alpha_n}, \ldots, S_{h^0_\kn,\wt K^0_n} , v_n)\}_{n \in \N} \subset \mathcal{C}^\alpha_{\mathbf{m}}$ and $(S_{h^\alpha,K^\alpha}, \ldots ,S_{h^0,K^0}, u) \in \mathcal{C}^\alpha_{{ \sigma},\mathbf{m}}$ such that 
	$$(S_{h^\alpha_\kn,\wt K^\alpha_n}, \ldots, S_{h^0_\kn,\wt K^0_n} , v_n) \xrightarrow{\tau_{\mathcal{C}^\alpha}} (S_{h^\alpha,K^\alpha}, \ldots ,S_{h^0,K^0}, u)$$
	 as $n \to \infty$ and
		\begin{equation}
			\label{mleq:existence1}
			\liminf_{n\to \infty}{\mathcal{F}^\alpha_{\sigma}(S_{h^\alpha_\kn,\wt K^\alpha_n}, \ldots, S_{h^0_\kn,\wt K^0_n} , v_n)} = \liminf_{n \to \infty}{\mathcal{F}^\alpha_{ \sigma} (S_{h^\alpha_\kn,K^\alpha_\kn}, \ldots ,S_{h^0_\kn,K^0_\kn}, u_\kn)}.
		\end{equation} 
		According to Theorem \ref{mlthm:lowersemicontinuitymultilayer}, we have that
		\begin{equation}
			\label{mleq:existance2}
			\mathcal{F}^\alpha_{ \sigma}(S_{h^\alpha,K^\alpha}, \ldots ,S_{h^0,K^0}, u) \le \liminf_{n \to \infty}{\mathcal{F}^\alpha_{\sigma}(S_{h^\alpha_\kn,\wt K^\alpha_n}, \ldots, S_{h^0_\kn,\wt K^0_n} ,v_n)}.
		\end{equation}
		We claim that for every $i =0, \ldots, \alpha $,  $S_{h^i_\kn, \wt K^i_n}$ and $S_{h^i, \wt K^i}$ satisfy the volume constraints of \eqref{mleq:const}. Indeed, by Theorem \ref{mlcompactness:Clambda}, for any $n  \in \N$, $\mathbbm{v}_i = \mathcal{L}^2({S_{h^i_\kn, K^i_\kn}} )=  \mathcal{L}^2({S_{h^i_\kn, \wt K^i_n}})$ for every $i =0, \ldots, \alpha $. Fix $i =0, \ldots,\alpha $. By definition of $\tau_{\B^\alpha}$-convergence and by 		
applying \cite[Lemma 3.2]{KP} we infer that $S_{h^i_\kn, \wt K^i_n} \to S_{h^i,K^i}$ in $L^1(\R^2)$ as $n \to \infty$, and thus $\mathcal{L}^2({S_{h^i,K^i}}) = \mathbbm{v}_i$. From \eqref{mleq:existence1} and \eqref{mleq:existance2}, we deduce that
		\begin{equation*}
			\begin{split}
				&\inf_{\small \begin{split}  
				(&S_{h^\alpha,K^\alpha}, \ldots ,S_{h^0,K^0}, u) \in \mathcal{C}^\alpha_{\mathbf{m}}, \\ & \mathcal{L}^2({S_{h^i,K^i}})= \mathbbm{v}_i,\, i = 0, \ldots, \alpha
				\end{split}
				 }{\mathcal{F}^\alpha _{\sigma}(S_{h^\alpha,K^\alpha}, \ldots ,S_{h^0,K^0}, u)} \\
    & \qquad = \lim_{n \to \infty}{\mathcal{F}^\alpha_{{\sigma}}(S_{h^\alpha_\kn,\wt K^\alpha_n}, \ldots, S_{h^0_\kn,\wt K^0_n} , u_{k_n})}\\
				& \qquad \ge \liminf_{n \to \infty}{\mathcal{F}^\alpha_{\sigma}(S_{h^\alpha_\kn,\wt K^\alpha_n}, \ldots, S_{h^0_\kn,\wt K^0_n} , u_{k_n})} \ge \mathcal{F}^\alpha_{\sigma}(S_{h^\alpha,K^\alpha}, \ldots ,S_{h^0,K^0}, u).
			\end{split}
		\end{equation*}
		We conclude from the previous inequality  that $(S_{h^\alpha,K^\alpha}, \ldots ,S_{h^0,K^0}, u)$ is a minimum of \eqref{mleq:const}. 
  
  The same strategy is   used to solve the unconstrained problem \eqref{mleq:uncost} thanks to the  extra observation that   for any minimizing sequence $\{(S_{h^\alpha_k,K^\alpha_k}, \ldots ,S_{h^0_k,K^0_k}, u_k)\} \subset \mathcal{C}^\alpha_{\mathbf{m}}$ of $\mathcal{F}^{\alpha,\bm{\lambda}}_{\sigma}$ such that 
		$$\sup_{k\in\mathbb{N}} {\mathcal{F}^{\alpha,\bm{\lambda}}_{\sigma}(S_{h^\alpha_k,K^\alpha_k}, \ldots ,S_{h^0_k,K^0_k}, u_k)}< \infty$$
  we have that
  $$
    \mathcal{L}^2(S_{h^\alpha_k,K^\alpha_k}) \le \abs{\mathcal{L}^2(S_{h^\alpha_k,K^\alpha_k}) - \mathbbm{v}_\alpha} +\mathbbm{v}_\alpha \le \frac1{\lambda_\alpha} \mathcal{F}^{\alpha,\bm{\lambda}}(S_{h^\alpha_k,K^\alpha_k}, \ldots ,S_{h^0_k,K^0_k}, u_k) +\mathbbm{v}_\alpha
  $$
  for every $k \in \N$.

  This concludes the proof. 
	\end{proof}


 \section*{Acknowledgments} 
 The authors acknowledge the support received from the Austrian Science Fund (FWF) projects P 29681 and TAI 293, from the Vienna Science and Technology Fund (WWTF) together with the City of Vienna and Berndorf Privatstiftung through Project MA16-005, and from BMBWF through the OeAD-WTZ project HR 08/2020.  R. Llerena thanks for the support obtained by a French public grant as part of the ``Investissement d'avenir'' project ANR-11-LABX-0056-LMH, LabEx LMH. P. Piovano is member of the Italian ``Gruppo Nazionale per l'Analisi Matematica, la Probabilit\`a e le loro Applicazioni'' (GNAMPA) and has received funding from the GNAMPA-INdAM 2022 project CUP: E55F22000270001 and 2023 Project CUP: E53C22001930001. P. Piovano also acknowledges the support obtained by the Italian Ministry of University and Research  (MUR) through the PRIN Project ``Partial differential equations and related geometric-functional inequalities''.  Finally,  P. Piovano is also grateful for the support received as \emph{Visiting Professor and Excellence Chair} and from the \emph{Visitor in the Theoretical Visiting Sciences Program} (TSVP) at the Okinawa Institute of Science and Technology (OIST), Japan. 
	

	\nocite{*}
	\bibliographystyle{amsplain}
	\bibliography{reference4.bib}

\providecommand{\bysame}{\leavevmode\hbox to3em{\hrulefill}\thinspace}
\providecommand{\MR}{\relax\ifhmode\unskip\space\fi MR }
\providecommand{\MRhref}[2]{%
  \href{http://www.ams.org/mathscinet-getitem?mr=#1}{#2}
}
\providecommand{\href}[2]{#2}
\begin{thebibliography}{10}

\bibitem{Algrem}
F.~J. Almgren, \emph{Existence and regularity almost everywhere of solutions to
  elliptic variational problems with constraints}, Mem. Amer. Math. Soc.
  \textbf{4-165} (1976), viii+199.

\bibitem{AB}
L.~Ambrosio and A.~Braides, \emph{Functionals defined on partitions in sets of
  finite perimeter {I}: integral}, J. Math. Pures Appl. \textbf{69} (1990),
  285--305.

\bibitem{AB2}
\bysame, \emph{Functionals defined on partitions in sets of finite perimeter
  {II}: semicontinuity}, J. Math. Pures Appl. \textbf{69} (1990), 307--333.

\bibitem{AFP}
L.~Ambrosio, N.~Fusco, and D.~Pallara, \emph{{Functions of bounded variation
  and free discontinuity problems}}, Oxford University Press, New York, 2000.

\bibitem{andrei}
E.-Y. {Andrei} and A.-H. {MacDonald}, \emph{Graphene bilayers with a twist},
  Nat. Mater \textbf{19} (2020), no.~12, 1265--1275.

\bibitem{AT}
R.~Asaro and W.~Tiller, \emph{Interface morphology development during stress
  corrosion cracking: Part {I}. {V}ia surface diffusion}, Metall. Trans.
  \textbf{3} (1972), 1789--1796.

\bibitem{BonCris}
M.~Bonacini and R.~Cristoferi, \emph{Area {Q}uasi-minimizing {P}artitions with
  a {G}raphical {C}onstraint: {R}elaxation and {T}wo-{D}imensional {P}artial
  {R}egularity}, J Nonlinear Sci \textbf{32} (2022), no.~6, 93.

\bibitem{BFG}
D.~Bucur, I.~Fragal{\`a}, and A.~Giacomini, \emph{The {M}ultiphase
  {M}umford-{S}hah problem}, SIAM J. Imaging Sci. \textbf{12} (2019),
  1561--1583.

\bibitem{BFG2}
\bysame, \emph{Multiphase free discontinuity problems: Monotonicity formula and
  regularity results}, Ann. Inst. Henri Poincare (C) Anal. Non Lineaire
  \textbf{38} (2021), 1553--1582.

\bibitem{cara1}
D.~Caraballo, \emph{The triangle inequalities and lower semi-continuity of
  surface energy of partitions}, Proc. R. Soc. Edinb. A: Math. \textbf{139}
  (2009), no.~3, 449--457.

\bibitem{cara2}
\bysame, \emph{{BV}-{E}llipticity and {L}ower {S}emicontinuity of {S}urface
  {E}nergy of {C}accioppoli {P}artitions of $\mathbb{R}^n$}, J Geom Anal
  \textbf{23} (2013), 202--220.

\bibitem{cara3}
\bysame, \emph{Existence of surface energy minimizing partitions of $\mathbb
  {R}^{n}$ satisfying volume constraints}, Trans. Amer. Math. Soc. \textbf{369}
  (2016), 1517--1546.

\bibitem{cermelligurtin1}
P.~Cermelli and M.E. Gurtin, \emph{The dynamics of solid-solid phase
  transitions 2. {I}ncoherent interfaces}, Arch. Rational Mech. Anal.
  \textbf{127} (1994), no.~1, 41--99.

\bibitem{ChB}
A.~Chambolle and E.~Bonnetier, \emph{Computing the equilibrium configuration of
  epitaxially strained crystalline films}, SIAM J. Appl. Math. \textbf{62}
  (2002), 1093--1121.

\bibitem{CC1}
A.~Chambolle and V.~Crismale, \emph{Existence of strong solutions to the
  {D}irichlet problem for the {G}riffith energy}, Calc. Var. Partial
  Differential Equations \textbf{58} (2019), 136.

\bibitem{Chirranjeevi}
B.G. Chirranjeevi, T.A. Abinandanan, and M.P. Gururajan, \emph{A phase field
  study of morphological instabilities in multilayer thin films}, Acta Mater.
  \textbf{57} (2009), no.~4, 1060--1067.

\bibitem{CTV}
M.~Conti, S.~Terracini, and G.~Verzini, \emph{A variational problem for the
  spatial segregation of reaction-diffusion systems}, Indiana Univ. Math. J.
  \textbf{54-3} (2005), 779--815.

\bibitem{DMS1}
G.~{Dal Maso}, J.~{Morel}, and S.~{Solimini}, \emph{A variational method in
  image segmentation: existence and approximation results}, Acta Math.
  \textbf{168} (1992), 89--151.

\bibitem{D}
A~Danescu, \emph{The {A}saro--{T}iller--{G}rinfeld instability revisited}, Int.
  J. Solids Struct. \textbf{38} (2001), 4671--4684.

\bibitem{DP}
E.~{Davoli} and P.~{Piovano}, \emph{Analytical validation of the
  {Y}oung--{D}upr\'e law for epitaxially-strained thin films}, Math. Models
  Methods Appl. Sci \textbf{29} (2019), 2183--2223.

\bibitem{DP2}
\bysame, \emph{Derivation of a heteroepitaxial thin-film model}, Interfaces
  Free Bound. \textbf{22} (2020), 1--26.

\bibitem{DphM}
G.~De~Philippis and F.~Maggi, \emph{Regularity of {F}ree {B}oundaries in
  {A}nisotropic {C}apillarity {P}roblems and the {V}alidity of {Y}oung's
  {L}aw}, Arch. Ration. Mech. Anal. \textbf{216} (2015), 473--568.

\bibitem{EG}
L.C. Evans and R.F. Gariepy, \emph{Measure {T}heory and {F}ine {P}roperties of
  {F}unctions}, Taylor \& Francis, New York, 1991.

\bibitem{F}
K.~Falconer, \emph{The geometry of fractal sets}, no.~85, Cambridge
  {U}niversity {P}ress, Cambridge, 1986.

\bibitem{FFLM}
I.~Fonseca, N.~Fusco, G.~Leoni, and V.~Millot, \emph{Material voids in elastic
  solids with anisotropic surface energies}, J. Math. Pures Appl. \textbf{96}
  (2011), 591--639.

\bibitem{FFLM2}
I.~Fonseca, N.~Fusco, G.~Leoni, and M.~Morini, \emph{Equilibrium configurations
  of epitaxially strained crystalline films: existence and regularity results},
  Arch. Ration. Mech. Anal. \textbf{186} (2007), 477--537.

\bibitem{FL}
I.~Fonseca and G.~Leoni, \emph{Modern methods in the {C}alculus of
  {V}ariations: {$L^p$}-spaces}, Springer, New York, 2007.

\bibitem{FrMa}
G.A. Francfort and J.-J. Marigo, \emph{Revisiting brittle fracture as an energy
  minimization problem}, Mech. Phys. Solids \textbf{46} (1998), 1319--1342.

\bibitem{friedG}
E.~Fried and M.~Gurtin, \emph{A unified treatment of evolving interfaces
  accounting for small deformations and atomic transport with emphasis on
  grain-boundaries and epitaxy}, Adv. Appl. Mech. \textbf{40} (2004), 1--177.

\bibitem{friesolo}
M.~Friedrich, M.~Perugini, and F.~Solombrino, \emph{Lower semicontinuity for
  functionals defined on piecewise rigid functions and on {GSBD}}, J. Funct.
  Anal. \textbf{280} (2021), 108929.

\bibitem{G}
A.~Giacomini, \emph{A generalization of {G}o{\l}\k ab's theorem and
  applications to fracture mechanics}, Math. Models Methods Appl. Sci.
  \textbf{12} (2002), 1245--1267.

\bibitem{Gr}
M.A Grinfeld, \emph{The stress driven instability in elastic crystals:
  Mathematical models and physical manifestations}, J. Nonlinear Sci.
  \textbf{3} (1993), 35--83.

\bibitem{Gurtin1}
M.E. Gurtin, \emph{The dynamics of solid-solid phase transitions 1. {C}oherent
  interfaces}, Arch. Rational Mech. Anal. \textbf{123} (1993), no.~4, 305--335.

\bibitem{HuangDesai}
Z.-F. Huang and R.~C. Desai, \emph{Stress-driven instability in growing
  multilayer films}, Phys. Rev. B \textbf{67} (2003), 075416.

\bibitem{KP}
Sh. {Kholmatov} and P.~{Piovano}, \emph{{A Unified Model for Stress-Driven
  Rearrangement Instabilities}}, Arch. Rational Mech. Anal. (2020), 415--488.

\bibitem{KP1}
\bysame, \emph{Existence of minimizers for the {SDRI} model in 2d: wetting and
  dewetting regime with mismatch strain}, Adv. Calc. (2023).

\bibitem{KP2}
\bysame, \emph{Existence of minimizers for the {SDRI} model in $\mathbb{R}^n$:
  Wetting and dewetting regimes with mismatch strain}, arXiv: Analysis of PDEs
  (2023).

\bibitem{KreutzP}
L.~Kreutz and P.~Piovano, \emph{{Microscopic Validation of a Variational Model
  of Epitaxially Strained Crystalline Films}}, SIAM J. Appl. Math. \textbf{53}
  (2021), 453--490.

\bibitem{Tersoff3}
F.~K. LeGoues, P.~M. Mooney, and J.~Tersoff, \emph{Measurement of the
  activation barrier to nucleation of dislocations in thin films}, Phys. Rev.
  Lett. \textbf{71} (1993), 396--399.

\bibitem{Baldelli:2014}
A.A. {Le\'on Baldelli}, J.-F. {Babadjian}, B.~{Bourdin}, D.~{Henao}, and
  C.~{Maurini}, \emph{A variational model for fracture and debonding of thin
  films under in-plane loadings}, J. Mech. Phys. Solids \textbf{70} (2014),
  320--348.

\bibitem{L}
G.~Leoni, \emph{A first course in {S}obolev spaces}, Graduate studies in
  mathematics; 105, American Math. Soc., Providence, 2009.

\bibitem{LlP}
R.~Llerena and P.~Piovano, \emph{Existence of minimizers for a two-phase free
  boundary problem with coherent and incoherent interfaces}, Submitted (2023).

\bibitem{M}
F.~Maggi, \emph{{Sets of Finite Perimeter and Geometric Variational Problems:
  An Introduction to Geometric Measure Theory}}, Cambridge University Press,
  2012.

\bibitem{Mp}
P.~Mattila, \emph{{Geometry of sets and measures in Euclidean spaces: fractals
  and rectifiability}}, Cambridge University Press, Cambridge, 1999.

\bibitem{Morgan}
F.~Morgan, \emph{Lowersemicontinuity of energy clusters}, Proc. R. Soc. Edinb.
  A: Math. \textbf{127} (1997), 819--822.

\bibitem{MS}
D.~Mumford and J.~Shah, \emph{Optimal approximations by piecewise smooth
  functions and associated variational problems}, Comm. Pure Appl. Math.
  \textbf{42-5} (1989), 577--685.

\bibitem{Prentice}
J.~Prentice, \emph{Coherent, partially coherent and incoherent light absorption
  in thin-film multilayer structures}, J. Phys. D: Appl. Phys. \textbf{33}
  (2000), no.~24, 3139.

\bibitem{spencer1}
B.~J. Spencer, \emph{Asymptotic derivation of the glued-wetting-layer model and
  contact-angle condition for stranski-krastanow islands}, Phys. Rev. B
  \textbf{59} (1999), 2011--2017.

\bibitem{spencer2}
B.~J. Spencer and J.~Tersoff, \emph{Equilibrium shapes and properties of
  epitaxially strained islands}, Phys. Rev. Lett. \textbf{79} (1997),
  4858--4861.

\bibitem{SRS}
N.~Sridhar, J.~M. Rickman, and D.~J. Srolovitz, \emph{{Multilayer film
  stability}}, J. Appl. Phys. \textbf{82} (1997), no.~10, 4852--4859.

\bibitem{S}
D.J Srolovitz, \emph{On the stability of surfaces of stressed solids}, Acta
  Metal. \textbf{37} (1989), 621--625.

\bibitem{SroloviztAnderson}
D.J. Srolovitz and M.P. Anderson, \emph{A criterion for compressive failure of
  a continuous, protective surface film}, Acta Metall. \textbf{32} (1984),
  no.~7, 1089--1092.

\bibitem{Tersoff2}
C.~Teichert, M.~G. Lagally, L.~J. Peticolas, J.~C. Bean, and J.~Tersoff,
  \emph{Stress-induced self-organization of nanoscale structures in sige/si
  multilayer films}, Phys. Rev. B \textbf{53} (1996), 16334--16337.

\bibitem{Tersoff1}
J.~Tersoff, C.~Teichert, and M.~G. Lagally, \emph{Self-organization in growth
  of quantum dot superlattices}, Phys. Rev. Lett. \textbf{76} (1996),
  1675--1678.

\end{thebibliography}
\end{document}